\documentclass[10pt]{amsart}
\usepackage{qtree,array,youngtab}
\usepackage{color, graphicx, comment}
\usepackage{tikz}

\definecolor{darkgreen}{rgb}{0.,0.5,0.}

\newcommand{\la}{\lambda}

\allowdisplaybreaks

\numberwithin{equation}{section} \overfullrule 5pt
\newtheorem{thm}{Theorem}[section]

\newtheorem{lem}[thm]{Lemma}

\theoremstyle{definition}
\newtheorem{defi}{Definition}[section]

\newcommand{\PP}[1]{\text{\rm #1}}

\title[Doubled shifted plane partitions]{%
Skew doubled shifted plane partitions: calculus and asymptotics}
\author{Guo-Niu HAN and Huan XIONG}
\address{Universit\'e de Strasbourg, CNRS, IRMA UMR 7501, F-67000 Strasbourg, France}
\email{guoniu.han@unistra.fr, \quad xiong@math.unistra.fr}



\begin{document}
\begin{abstract} % <<<
Plane partitions have been widely studied in Mathematics since MacMahon. See, for example, the works by Andrews, Macdonald, Stanley, Sagan and Krattenthaler. The Schur process approach, introduced by Okounkov and Reshetikhin, and further developed by Borodin, Corwin,  Corteel, Savelief and Vuleti\'c, has been proved to be a powerful tool in the study of various kinds of plane partitions. The exact enumerations of ordinary plane partitions, shifted plane partitions and cylindric partitions could be derived from two summation formulas for Schur processes, namely, the {\it open summation formula} and {\it the cylindric summation formula}.	

In this paper, we establish a new summation formula for Schur processes, called the {\it complete summation formula}. As an application, we obtain the generating function and the asymptotic formula for the number of {\it doubled shifted plane partitions}, which can be viewed as plane partitions ``shifted at the two sides''. We prove that the order of the asymptotic formula depends only on the diagonal width of the doubled shifted plane partition, not on the profile (the skew zone) itself. By using the same methods, the generating function and the asymptotic formula for the number of {\it symmetric cylindric partitions} are also derived.
\end{abstract}

% >>>
\maketitle

\section{Introduction} % <<<

%\textbf{TODO: ribbon to double shifted}

\medskip

An {\it ordinary plane partition} (resp. A {\it defective plane partition}) 
is a filling $\omega=(\omega_{i,j})$ of
the quarter plane $\Lambda=\{(i,j)\mid i,j\geq 1\}$ (resp. of a 
connected area of the quarter plane $\Lambda$) with
nonnegative integers
such that rows and columns decrease weakly, and
the size $|\omega|=\sum \omega_{i,j}$ is finite.
The enumeration of
various defective plane partitions 
have been widely studied  (see \cite{Andrews1978PP1,GesselKratt1997,GordonHouten1968PP12,MillsRR1982,MillsRR1983,MillsRR1986,MillsRR1987,Stanley1971,Stanley1973,Stanley1986SPP,Stanely1987Err,Stanley1989}). In particular,
the generating functions for the following five types of defective plane partitions (see Fig. 1) have been obtained since MacMahon:

(A) the ordinary plane partitions (MacMahon \cite{MacMahon1899,MacMahon1916}, Stanley \cite{Stanley1971});

(B) the skew plane partitions (Sagan \cite{Sagan1993});

(C) the skew shifted plane partitions  (Sagan \cite{Sagan1993});

(D) the symmetric plane partitions  (Andrews \cite{Andrews1978PP1}, Macdonald \cite{Macdonald1995});

(E) the cylindric partitions (Gessel and Krattenthaler \cite{GesselKratt1997},  Borodin  \cite{Borodin2007}).

\smallskip

In the literature there are several approaches to deriving the generating functions for various defective plane partitions, such as: 
(1)~Determinant evaluation and nonintersecting lattice paths  \cite{Andrews1977,Andrews1978PP1,Andrews1979PP3,Andrews1994PP5,AndrewsPS2005PP6,GesselKratt1997,Kratt1990};
(2)~Hook lengths and combinatorial proofs \cite{ Sagan1982,Sagan1993,Stanley1973};
(3)~Schur functions and Schur processes \cite{betea2017free, Okoun2006,Panova2012,Stanley1971}.
(4)~Lozenge tilings and Kuo condensation \cite{ciucu2002enumeration, CiucuKratt2013}.

$$
\begin{tikzpicture}[scale=0.6]
\begin{scope}[xshift=0, yshift=0]
\fill [gray!30](0.0000,0.2000)--(0.0000,0.4000)--(0.0000,0.6000)--(0.0000,0.8000)--(0.0000,1.0000)--(0.0000,1.2000)--(0.0000,1.4000)--(0.0000,1.6000)--(0.0000,1.8000)--(0.0000,2.0000)--(0.0000,2.2000)--(0.0000,2.4000)--(0.0000,2.6000)--(0.0000,2.8000)--(0.0000,3.0000)--(0.0000,3.2000)--(0.0000,3.4000)--(0.2000,3.4000)--(0.4000,3.4000)--(0.6000,3.4000)--(0.8000,3.4000)--(1.0000,3.4000)--(1.2000,3.4000)--(1.4000,3.4000)--(1.6000,3.4000)--(1.8000,3.4000)--(2.0000,3.4000)--(2.2000,3.4000)--(2.4000,3.4000)--(2.6000,3.4000)--(2.8000,3.4000)--(3.0000,3.4000)--(3.2000,3.4000)--(3.4000,3.4000)--(3.6000,3.4000)--(3.8000,3.4000)--(4.0000,3.4000)--(4.0000,3.2000)--(4.0000,3.0000)--(4.0000,2.8000)--(4.0000,2.6000)--(4.0000,2.4000)--(4.0000,2.2000)--(4.0000,2.0000)--(4.0000,1.8000)--(4.0000,1.6000)--(4.0000,1.4000)--(4.0000,1.2000)--(4.0000,1.0000)--(4.0000,0.8000)--(4.0000,0.6000)--(4.0000,0.4000)--(4.0000,0.2000)--(3.8000,0.2000)--(3.6000,0.2000)--(3.4000,0.2000)--(3.2000,0.2000)--(3.0000,0.2000)--(2.8000,0.2000)--(2.6000,0.2000)--(2.4000,0.2000)--(2.2000,0.2000)--(2.0000,0.2000)--(1.8000,0.2000)--(1.6000,0.2000)--(1.4000,0.2000)--(1.2000,0.2000)--(1.0000,0.2000)--(0.8000,0.2000)--(0.6000,0.2000)--(0.4000,0.2000)--(0.2000,0.2000)--(0.0000,0.2000);
\draw [gray!10](0.0000,0.2000)--(0.0000,3.4000);
\draw [gray!10](0.2000,0.2000)--(0.2000,3.4000);
\draw [gray!10](0.4000,0.2000)--(0.4000,3.4000);
\draw [gray!10](0.6000,0.2000)--(0.6000,3.4000);
\draw [gray!10](0.8000,0.2000)--(0.8000,3.4000);
\draw [gray!10](1.0000,0.2000)--(1.0000,3.4000);
\draw [gray!10](1.2000,0.2000)--(1.2000,3.4000);
\draw [gray!10](1.4000,0.2000)--(1.4000,3.4000);
\draw [gray!10](1.6000,0.2000)--(1.6000,3.4000);
\draw [gray!10](1.8000,0.2000)--(1.8000,3.4000);
\draw [gray!10](2.0000,0.2000)--(2.0000,3.4000);
\draw [gray!10](2.2000,0.2000)--(2.2000,3.4000);
\draw [gray!10](2.4000,0.2000)--(2.4000,3.4000);
\draw [gray!10](2.6000,0.2000)--(2.6000,3.4000);
\draw [gray!10](2.8000,0.2000)--(2.8000,3.4000);
\draw [gray!10](3.0000,0.2000)--(3.0000,3.4000);
\draw [gray!10](3.2000,0.2000)--(3.2000,3.4000);
\draw [gray!10](3.4000,0.2000)--(3.4000,3.4000);
\draw [gray!10](3.6000,0.2000)--(3.6000,3.4000);
\draw [gray!10](3.8000,0.2000)--(3.8000,3.4000);
\draw [gray!10](0.0000,0.2000)--(4.0000,0.2000);
\draw [gray!10](0.0000,0.4000)--(4.0000,0.4000);
\draw [gray!10](0.0000,0.6000)--(4.0000,0.6000);
\draw [gray!10](0.0000,0.8000)--(4.0000,0.8000);
\draw [gray!10](0.0000,1.0000)--(4.0000,1.0000);
\draw [gray!10](0.0000,1.2000)--(4.0000,1.2000);
\draw [gray!10](0.0000,1.4000)--(4.0000,1.4000);
\draw [gray!10](0.0000,1.6000)--(4.0000,1.6000);
\draw [gray!10](0.0000,1.8000)--(4.0000,1.8000);
\draw [gray!10](0.0000,2.0000)--(4.0000,2.0000);
\draw [gray!10](0.0000,2.2000)--(4.0000,2.2000);
\draw [gray!10](0.0000,2.4000)--(4.0000,2.4000);
\draw [gray!10](0.0000,2.6000)--(4.0000,2.6000);
\draw [gray!10](0.0000,2.8000)--(4.0000,2.8000);
\draw [gray!10](0.0000,3.0000)--(4.0000,3.0000);
\draw [gray!10](0.0000,3.2000)--(4.0000,3.2000);
\draw [black](0.0000,0.2000)--(0.0000,0.4000)--(0.0000,0.6000)--(0.0000,0.8000)--(0.0000,1.0000)--(0.0000,1.2000)--(0.0000,1.4000)--(0.0000,1.6000)--(0.0000,1.8000)--(0.0000,2.0000)--(0.0000,2.2000)--(0.0000,2.4000)--(0.0000,2.6000)--(0.0000,2.8000)--(0.0000,3.0000)--(0.0000,3.2000)--(0.0000,3.4000)--(0.2000,3.4000)--(0.4000,3.4000)--(0.6000,3.4000)--(0.8000,3.4000)--(1.0000,3.4000)--(1.2000,3.4000)--(1.4000,3.4000)--(1.6000,3.4000)--(1.8000,3.4000)--(2.0000,3.4000)--(2.2000,3.4000)--(2.4000,3.4000)--(2.6000,3.4000)--(2.8000,3.4000)--(3.0000,3.4000)--(3.2000,3.4000)--(3.4000,3.4000)--(3.6000,3.4000)--(3.8000,3.4000)--(4.0000,3.4000);
\draw (0.0000,3.4000)--(0.0000,-0.2000);
\draw (0.0000,3.4000)--(4.4000,3.4000);
\draw (2,-1) node [] {A: Ordinary PP};
\end{scope}
\begin{scope}[xshift=160, yshift=0]
\fill [gray!30](0.0000,0.2000)--(0.0000,0.4000)--(0.0000,0.6000)--(0.0000,0.8000)--(0.0000,1.0000)--(0.0000,1.2000)--(0.0000,1.4000)--(0.0000,1.6000)--(0.0000,1.8000)--(0.0000,2.0000)--(0.0000,2.2000)--(0.0000,2.4000)--(0.0000,2.6000)--(0.0000,2.8000)--(0.2000,2.8000)--(0.4000,2.8000)--(0.4000,3.0000)--(0.6000,3.0000)--(0.8000,3.0000)--(1.0000,3.0000)--(1.0000,3.2000)--(1.2000,3.2000)--(1.2000,3.4000)--(1.4000,3.4000)--(1.6000,3.4000)--(1.8000,3.4000)--(2.0000,3.4000)--(2.2000,3.4000)--(2.4000,3.4000)--(2.6000,3.4000)--(2.8000,3.4000)--(3.0000,3.4000)--(3.2000,3.4000)--(3.4000,3.4000)--(3.6000,3.4000)--(3.8000,3.4000)--(3.8000,3.2000)--(3.8000,3.0000)--(3.8000,2.8000)--(3.8000,2.6000)--(3.8000,2.4000)--(3.8000,2.2000)--(3.8000,2.0000)--(3.8000,1.8000)--(3.8000,1.6000)--(3.8000,1.4000)--(3.8000,1.2000)--(3.8000,1.0000)--(3.8000,0.8000)--(3.8000,0.6000)--(3.8000,0.4000)--(3.8000,0.2000)--(3.6000,0.2000)--(3.4000,0.2000)--(3.2000,0.2000)--(3.0000,0.2000)--(2.8000,0.2000)--(2.6000,0.2000)--(2.4000,0.2000)--(2.2000,0.2000)--(2.0000,0.2000)--(1.8000,0.2000)--(1.6000,0.2000)--(1.4000,0.2000)--(1.2000,0.2000)--(1.0000,0.2000)--(0.8000,0.2000)--(0.6000,0.2000)--(0.4000,0.2000)--(0.2000,0.2000);
\draw [gray!10](0.0000,0.2000)--(0.0000,2.8000);
\draw [gray!10](0.2000,0.2000)--(0.2000,2.8000);
\draw [gray!10](0.4000,0.2000)--(0.4000,3.0000);
\draw [gray!10](0.6000,0.2000)--(0.6000,3.0000);
\draw [gray!10](0.8000,0.2000)--(0.8000,3.0000);
\draw [gray!10](1.0000,0.2000)--(1.0000,3.2000);
\draw [gray!10](1.2000,0.2000)--(1.2000,3.4000);
\draw [gray!10](1.4000,0.2000)--(1.4000,3.4000);
\draw [gray!10](1.6000,0.2000)--(1.6000,3.4000);
\draw [gray!10](1.8000,0.2000)--(1.8000,3.4000);
\draw [gray!10](2.0000,0.2000)--(2.0000,3.4000);
\draw [gray!10](2.2000,0.2000)--(2.2000,3.4000);
\draw [gray!10](2.4000,0.2000)--(2.4000,3.4000);
\draw [gray!10](2.6000,0.2000)--(2.6000,3.4000);
\draw [gray!10](2.8000,0.2000)--(2.8000,3.4000);
\draw [gray!10](3.0000,0.2000)--(3.0000,3.4000);
\draw [gray!10](3.2000,0.2000)--(3.2000,3.4000);
\draw [gray!10](3.4000,0.2000)--(3.4000,3.4000);
\draw [gray!10](3.6000,0.2000)--(3.6000,3.4000);
\draw [gray!10](0.0000,0.2000)--(3.8000,0.2000);
\draw [gray!10](0.0000,0.4000)--(3.8000,0.4000);
\draw [gray!10](0.0000,0.6000)--(3.8000,0.6000);
\draw [gray!10](0.0000,0.8000)--(3.8000,0.8000);
\draw [gray!10](0.0000,1.0000)--(3.8000,1.0000);
\draw [gray!10](0.0000,1.2000)--(3.8000,1.2000);
\draw [gray!10](0.0000,1.4000)--(3.8000,1.4000);
\draw [gray!10](0.0000,1.6000)--(3.8000,1.6000);
\draw [gray!10](0.0000,1.8000)--(3.8000,1.8000);
\draw [gray!10](0.0000,2.0000)--(3.8000,2.0000);
\draw [gray!10](0.0000,2.2000)--(3.8000,2.2000);
\draw [gray!10](0.0000,2.4000)--(3.8000,2.4000);
\draw [gray!10](0.0000,2.6000)--(3.8000,2.6000);
\draw [gray!10](0.0000,2.8000)--(3.8000,2.8000);
\draw [gray!10](0.4000,3.0000)--(3.8000,3.0000);
\draw [gray!10](1.0000,3.2000)--(3.8000,3.2000);
\draw [black](0.0000,0.2000)--(0.0000,0.4000)--(0.0000,0.6000)--(0.0000,0.8000)--(0.0000,1.0000)--(0.0000,1.2000)--(0.0000,1.4000)--(0.0000,1.6000)--(0.0000,1.8000)--(0.0000,2.0000)--(0.0000,2.2000)--(0.0000,2.4000)--(0.0000,2.6000)--(0.0000,2.8000)--(0.2000,2.8000)--(0.4000,2.8000)--(0.4000,3.0000)--(0.6000,3.0000)--(0.8000,3.0000)--(1.0000,3.0000)--(1.0000,3.2000)--(1.2000,3.2000)--(1.2000,3.4000)--(1.4000,3.4000)--(1.6000,3.4000)--(1.8000,3.4000)--(2.0000,3.4000)--(2.2000,3.4000)--(2.4000,3.4000)--(2.6000,3.4000)--(2.8000,3.4000)--(3.0000,3.4000)--(3.2000,3.4000)--(3.4000,3.4000)--(3.6000,3.4000)--(3.8000,3.4000);
\draw (0.0000,3.4000)--(0.0000,-0.2000);
\draw (0.0000,3.4000)--(4.2000,3.4000);
\draw (2,-1) node [] {B: Skew PP};
\end{scope}
\begin{scope}[xshift=320, yshift=0]
\fill [gray!30](1.8000,0.2000)--(1.6000,0.2000)--(1.6000,0.4000)--(1.4000,0.4000)--(1.4000,0.6000)--(1.2000,0.6000)--(1.2000,0.8000)--(1.0000,0.8000)--(1.0000,1.0000)--(0.8000,1.0000)--(0.8000,1.2000)--(0.6000,1.2000)--(0.6000,1.4000)--(0.4000,1.4000)--(0.4000,1.6000)--(0.2000,1.6000)--(0.2000,1.8000)--(0.0000,1.8000)--(0.0000,2.0000)--(0.0000,2.2000)--(0.0000,2.4000)--(0.0000,2.6000)--(0.0000,2.8000)--(0.2000,2.8000)--(0.4000,2.8000)--(0.4000,3.0000)--(0.6000,3.0000)--(0.8000,3.0000)--(1.0000,3.0000)--(1.0000,3.2000)--(1.2000,3.2000)--(1.2000,3.4000)--(1.4000,3.4000)--(1.6000,3.4000)--(1.8000,3.4000)--(2.0000,3.4000)--(2.2000,3.4000)--(2.4000,3.4000)--(2.6000,3.4000)--(2.8000,3.4000)--(3.0000,3.4000)--(3.2000,3.4000)--(3.4000,3.4000)--(3.6000,3.4000)--(3.8000,3.4000)--(4.0000,3.4000)--(4.0000,3.2000)--(4.0000,3.0000)--(4.0000,2.8000)--(4.0000,2.6000)--(4.0000,2.4000)--(4.0000,2.2000)--(4.0000,2.0000)--(4.0000,1.8000)--(4.0000,1.6000)--(4.0000,1.4000)--(4.0000,1.2000)--(4.0000,1.0000)--(4.0000,0.8000)--(4.0000,0.6000)--(4.0000,0.4000)--(4.0000,0.2000)--(3.8000,0.2000)--(3.6000,0.2000)--(3.4000,0.2000)--(3.2000,0.2000)--(3.0000,0.2000)--(2.8000,0.2000)--(2.6000,0.2000)--(2.4000,0.2000)--(2.2000,0.2000)--(2.0000,0.2000)--(1.8000,0.2000);
\draw [gray!10](0.0000,1.8000)--(0.0000,2.8000);
\draw [gray!10](0.2000,1.6000)--(0.2000,2.8000);
\draw [gray!10](0.4000,1.4000)--(0.4000,3.0000);
\draw [gray!10](0.6000,1.2000)--(0.6000,3.0000);
\draw [gray!10](0.8000,1.0000)--(0.8000,3.0000);
\draw [gray!10](1.0000,0.8000)--(1.0000,3.2000);
\draw [gray!10](1.2000,0.6000)--(1.2000,3.4000);
\draw [gray!10](1.4000,0.4000)--(1.4000,3.4000);
\draw [gray!10](1.6000,0.2000)--(1.6000,3.4000);
\draw [gray!10](1.8000,0.2000)--(1.8000,3.4000);
\draw [gray!10](2.0000,0.2000)--(2.0000,3.4000);
\draw [gray!10](2.2000,0.2000)--(2.2000,3.4000);
\draw [gray!10](2.4000,0.2000)--(2.4000,3.4000);
\draw [gray!10](2.6000,0.2000)--(2.6000,3.4000);
\draw [gray!10](2.8000,0.2000)--(2.8000,3.4000);
\draw [gray!10](3.0000,0.2000)--(3.0000,3.4000);
\draw [gray!10](3.2000,0.2000)--(3.2000,3.4000);
\draw [gray!10](3.4000,0.2000)--(3.4000,3.4000);
\draw [gray!10](3.6000,0.2000)--(3.6000,3.4000);
\draw [gray!10](3.8000,0.2000)--(3.8000,3.4000);
\draw [gray!10](1.6000,0.2000)--(4.0000,0.2000);
\draw [gray!10](1.4000,0.4000)--(4.0000,0.4000);
\draw [gray!10](1.2000,0.6000)--(4.0000,0.6000);
\draw [gray!10](1.0000,0.8000)--(4.0000,0.8000);
\draw [gray!10](0.8000,1.0000)--(4.0000,1.0000);
\draw [gray!10](0.6000,1.2000)--(4.0000,1.2000);
\draw [gray!10](0.4000,1.4000)--(4.0000,1.4000);
\draw [gray!10](0.2000,1.6000)--(4.0000,1.6000);
\draw [gray!10](0.0000,1.8000)--(4.0000,1.8000);
\draw [gray!10](0.0000,2.0000)--(4.0000,2.0000);
\draw [gray!10](0.0000,2.2000)--(4.0000,2.2000);
\draw [gray!10](0.0000,2.4000)--(4.0000,2.4000);
\draw [gray!10](0.0000,2.6000)--(4.0000,2.6000);
\draw [gray!10](0.0000,2.8000)--(4.0000,2.8000);
\draw [gray!10](0.4000,3.0000)--(4.0000,3.0000);
\draw [gray!10](1.0000,3.2000)--(4.0000,3.2000);
\draw [black](1.8000,0.2000)--(1.6000,0.2000)--(1.6000,0.4000)--(1.4000,0.4000)--(1.4000,0.6000)--(1.2000,0.6000)--(1.2000,0.8000)--(1.0000,0.8000)--(1.0000,1.0000)--(0.8000,1.0000)--(0.8000,1.2000)--(0.6000,1.2000)--(0.6000,1.4000)--(0.4000,1.4000)--(0.4000,1.6000)--(0.2000,1.6000)--(0.2000,1.8000)--(0.0000,1.8000)--(0.0000,2.0000)--(0.0000,2.2000)--(0.0000,2.4000)--(0.0000,2.6000)--(0.0000,2.8000)--(0.2000,2.8000)--(0.4000,2.8000)--(0.4000,3.0000)--(0.6000,3.0000)--(0.8000,3.0000)--(1.0000,3.0000)--(1.0000,3.2000)--(1.2000,3.2000)--(1.2000,3.4000)--(1.4000,3.4000)--(1.6000,3.4000)--(1.8000,3.4000)--(2.0000,3.4000)--(2.2000,3.4000)--(2.4000,3.4000)--(2.6000,3.4000)--(2.8000,3.4000)--(3.0000,3.4000)--(3.2000,3.4000)--(3.4000,3.4000)--(3.6000,3.4000)--(3.8000,3.4000)--(4.0000,3.4000);
\draw (0.0000,3.4000)--(0.0000,-0.2000);
\draw (0.0000,3.4000)--(4.4000,3.4000);
\draw (2,-1) node [] {C: Skew Shifted PP};
\end{scope}
\end{tikzpicture}
$$

$$
\begin{tikzpicture}[scale=0.6]
\begin{scope}[xshift=0, yshift=0]
\fill [gray!30](0.0000,0.2000)--(0.0000,0.4000)--(0.0000,0.6000)--(0.0000,0.8000)--(0.0000,1.0000)--(0.0000,1.2000)--(0.0000,1.4000)--(0.0000,1.6000)--(0.0000,1.8000)--(0.0000,2.0000)--(0.0000,2.2000)--(0.0000,2.4000)--(0.0000,2.6000)--(0.0000,2.8000)--(0.0000,3.0000)--(0.0000,3.2000)--(0.0000,3.4000)--(0.2000,3.4000)--(0.4000,3.4000)--(0.6000,3.4000)--(0.8000,3.4000)--(1.0000,3.4000)--(1.2000,3.4000)--(1.4000,3.4000)--(1.6000,3.4000)--(1.8000,3.4000)--(2.0000,3.4000)--(2.2000,3.4000)--(2.4000,3.4000)--(2.6000,3.4000)--(2.8000,3.4000)--(3.0000,3.4000)--(3.2000,3.4000)--(3.2000,3.2000)--(3.2000,3.0000)--(3.2000,2.8000)--(3.2000,2.6000)--(3.2000,2.4000)--(3.2000,2.2000)--(3.2000,2.0000)--(3.2000,1.8000)--(3.2000,1.6000)--(3.2000,1.4000)--(3.2000,1.2000)--(3.2000,1.0000)--(3.2000,0.8000)--(3.2000,0.6000)--(3.2000,0.4000)--(3.2000,0.2000)--(3.0000,0.2000)--(2.8000,0.2000)--(2.6000,0.2000)--(2.4000,0.2000)--(2.2000,0.2000)--(2.0000,0.2000)--(1.8000,0.2000)--(1.6000,0.2000)--(1.4000,0.2000)--(1.2000,0.2000)--(1.0000,0.2000)--(0.8000,0.2000)--(0.6000,0.2000)--(0.4000,0.2000)--(0.2000,0.2000)--(0.0000,0.2000);
\draw [gray!10](0.0000,0.2000)--(0.0000,3.4000);
\draw [gray!10](0.2000,0.2000)--(0.2000,3.4000);
\draw [gray!10](0.4000,0.2000)--(0.4000,3.4000);
\draw [gray!10](0.6000,0.2000)--(0.6000,3.4000);
\draw [gray!10](0.8000,0.2000)--(0.8000,3.4000);
\draw [gray!10](1.0000,0.2000)--(1.0000,3.4000);
\draw [gray!10](1.2000,0.2000)--(1.2000,3.4000);
\draw [gray!10](1.4000,0.2000)--(1.4000,3.4000);
\draw [gray!10](1.6000,0.2000)--(1.6000,3.4000);
\draw [gray!10](1.8000,0.2000)--(1.8000,3.4000);
\draw [gray!10](2.0000,0.2000)--(2.0000,3.4000);
\draw [gray!10](2.2000,0.2000)--(2.2000,3.4000);
\draw [gray!10](2.4000,0.2000)--(2.4000,3.4000);
\draw [gray!10](2.6000,0.2000)--(2.6000,3.4000);
\draw [gray!10](2.8000,0.2000)--(2.8000,3.4000);
\draw [gray!10](3.0000,0.2000)--(3.0000,3.4000);
\draw [gray!10](0.0000,0.2000)--(3.2000,0.2000);
\draw [gray!10](0.0000,0.4000)--(3.2000,0.4000);
\draw [gray!10](0.0000,0.6000)--(3.2000,0.6000);
\draw [gray!10](0.0000,0.8000)--(3.2000,0.8000);
\draw [gray!10](0.0000,1.0000)--(3.2000,1.0000);
\draw [gray!10](0.0000,1.2000)--(3.2000,1.2000);
\draw [gray!10](0.0000,1.4000)--(3.2000,1.4000);
\draw [gray!10](0.0000,1.6000)--(3.2000,1.6000);
\draw [gray!10](0.0000,1.8000)--(3.2000,1.8000);
\draw [gray!10](0.0000,2.0000)--(3.2000,2.0000);
\draw [gray!10](0.0000,2.2000)--(3.2000,2.2000);
\draw [gray!10](0.0000,2.4000)--(3.2000,2.4000);
\draw [gray!10](0.0000,2.6000)--(3.2000,2.6000);
\draw [gray!10](0.0000,2.8000)--(3.2000,2.8000);
\draw [gray!10](0.0000,3.0000)--(3.2000,3.0000);
\draw [gray!10](0.0000,3.2000)--(3.2000,3.2000);
\draw [black](0.0000,0.2000)--(0.0000,0.4000)--(0.0000,0.6000)--(0.0000,0.8000)--(0.0000,1.0000)--(0.0000,1.2000)--(0.0000,1.4000)--(0.0000,1.6000)--(0.0000,1.8000)--(0.0000,2.0000)--(0.0000,2.2000)--(0.0000,2.4000)--(0.0000,2.6000)--(0.0000,2.8000)--(0.0000,3.0000)--(0.0000,3.2000)--(0.0000,3.4000)--(0.2000,3.4000)--(0.4000,3.4000)--(0.6000,3.4000)--(0.8000,3.4000)--(1.0000,3.4000)--(1.2000,3.4000)--(1.4000,3.4000)--(1.6000,3.4000)--(1.8000,3.4000)--(2.0000,3.4000)--(2.2000,3.4000)--(2.4000,3.4000)--(2.6000,3.4000)--(2.8000,3.4000)--(3.0000,3.4000)--(3.2000,3.4000);
\draw (0.0000,3.4000)--(0.0000,-0.2000);
\draw (0.0000,3.4000)--(3.6000,3.4000);
\draw [dashed] (0.0000,3.4000)--(3.4000,0.0000);
\draw (2,-1) node [] {D: Sym PP};
\end{scope}
\begin{scope}[xshift=160, yshift=0]
\fill [gray!30](1.6000,0.0000)--(1.4000,0.0000)--(1.4000,0.2000)--(1.2000,0.2000)--(1.2000,0.4000)--(1.0000,0.4000)--(1.0000,0.6000)--(0.8000,0.6000)--(0.8000,0.8000)--(0.6000,0.8000)--(0.6000,1.0000)--(0.4000,1.0000)--(0.4000,1.2000)--(0.2000,1.2000)--(0.2000,1.4000)--(0.0000,1.4000)--(0.0000,1.6000)--(0.0000,1.8000)--(0.0000,2.0000)--(0.0000,2.2000)--(0.0000,2.4000)--(0.2000,2.4000)--(0.2000,2.6000)--(0.4000,2.6000)--(0.4000,2.8000)--(0.6000,2.8000)--(0.8000,2.8000)--(0.8000,3.0000)--(0.8000,3.2000)--(1.0000,3.2000)--(1.2000,3.2000)--(1.4000,3.2000)--(1.6000,3.2000)--(1.6000,3.4000)--(1.8000,3.4000)--(2.0000,3.4000)--(2.0000,3.2000)--(2.2000,3.2000)--(2.2000,3.0000)--(2.4000,3.0000)--(2.4000,2.8000)--(2.6000,2.8000)--(2.6000,2.6000)--(2.8000,2.6000)--(2.8000,2.4000)--(3.0000,2.4000)--(3.0000,2.2000)--(3.2000,2.2000)--(3.2000,2.0000)--(3.4000,2.0000)--(3.4000,1.8000)--(3.2000,1.8000)--(3.2000,1.6000)--(3.0000,1.6000)--(3.0000,1.4000)--(2.8000,1.4000)--(2.8000,1.2000)--(2.6000,1.2000)--(2.6000,1.0000)--(2.4000,1.0000)--(2.4000,0.8000)--(2.2000,0.8000)--(2.2000,0.6000)--(2.0000,0.6000)--(2.0000,0.4000)--(1.8000,0.4000)--(1.8000,0.2000)--(1.6000,0.2000)--(1.6000,0.0000);
\draw [gray!10](0.0000,1.4000)--(0.0000,2.4000);
\draw [gray!10](0.2000,1.2000)--(0.2000,2.6000);
\draw [gray!10](0.4000,1.0000)--(0.4000,2.8000);
\draw [gray!10](0.6000,0.8000)--(0.6000,2.8000);
\draw [gray!10](0.8000,0.6000)--(0.8000,3.2000);
\draw [gray!10](1.0000,0.4000)--(1.0000,3.2000);
\draw [gray!10](1.2000,0.2000)--(1.2000,3.2000);
\draw [gray!10](1.4000,0.0000)--(1.4000,3.2000);
\draw [gray!10](1.6000,0.0000)--(1.6000,3.4000);
\draw [gray!10](1.8000,0.2000)--(1.8000,3.4000);
\draw [gray!10](2.0000,0.4000)--(2.0000,3.4000);
\draw [gray!10](2.2000,0.6000)--(2.2000,3.2000);
\draw [gray!10](2.4000,0.8000)--(2.4000,3.0000);
\draw [gray!10](2.6000,1.0000)--(2.6000,2.8000);
\draw [gray!10](2.8000,1.2000)--(2.8000,2.6000);
\draw [gray!10](3.0000,1.4000)--(3.0000,2.4000);
\draw [gray!10](3.2000,1.6000)--(3.2000,2.2000);
\draw [gray!10](1.4000,0.0000)--(1.6000,0.0000);
\draw [gray!10](1.2000,0.2000)--(1.8000,0.2000);
\draw [gray!10](1.0000,0.4000)--(2.0000,0.4000);
\draw [gray!10](0.8000,0.6000)--(2.2000,0.6000);
\draw [gray!10](0.6000,0.8000)--(2.4000,0.8000);
\draw [gray!10](0.4000,1.0000)--(2.6000,1.0000);
\draw [gray!10](0.2000,1.2000)--(2.8000,1.2000);
\draw [gray!10](0.0000,1.4000)--(3.0000,1.4000);
\draw [gray!10](0.0000,1.6000)--(3.2000,1.6000);
\draw [gray!10](0.0000,1.8000)--(3.4000,1.8000);
\draw [gray!10](0.0000,2.0000)--(3.4000,2.0000);
\draw [gray!10](0.0000,2.2000)--(3.2000,2.2000);
\draw [gray!10](0.0000,2.4000)--(3.0000,2.4000);
\draw [gray!10](0.2000,2.6000)--(2.8000,2.6000);
\draw [gray!10](0.4000,2.8000)--(2.6000,2.8000);
\draw [gray!10](0.8000,3.0000)--(2.4000,3.0000);
\draw [gray!10](0.8000,3.2000)--(2.2000,3.2000);
\draw [black](1.6000,0.0000)--(1.4000,0.0000)--(1.4000,0.2000)--(1.2000,0.2000)--(1.2000,0.4000)--(1.0000,0.4000)--(1.0000,0.6000)--(0.8000,0.6000)--(0.8000,0.8000)--(0.6000,0.8000)--(0.6000,1.0000)--(0.4000,1.0000)--(0.4000,1.2000)--(0.2000,1.2000)--(0.2000,1.4000)--(0.0000,1.4000)--(0.0000,1.6000)--(0.0000,1.8000)--(0.0000,2.0000)--(0.0000,2.2000)--(0.0000,2.4000)--(0.2000,2.4000)--(0.2000,2.6000)--(0.4000,2.6000)--(0.4000,2.8000)--(0.6000,2.8000)--(0.8000,2.8000)--(0.8000,3.0000)--(0.8000,3.2000)--(1.0000,3.2000)--(1.2000,3.2000)--(1.4000,3.2000)--(1.6000,3.2000)--(1.6000,3.4000)--(1.8000,3.4000)--(2.0000,3.4000)--(2.0000,3.2000)--(2.2000,3.2000)--(2.2000,3.0000)--(2.4000,3.0000)--(2.4000,2.8000)--(2.6000,2.8000)--(2.6000,2.6000)--(2.8000,2.6000)--(2.8000,2.4000)--(3.0000,2.4000)--(3.0000,2.2000)--(3.2000,2.2000)--(3.2000,2.0000)--(3.4000,2.0000)--(3.4000,1.8000);
\draw (0.0000,3.4000)--(0.0000,-0.2000);
\draw (0.0000,3.4000)--(3.6000,3.4000);
\draw (2.3,  2.4) node [] {$\lambda$};
\draw (1.1,  1.2) node [] {$\lambda$};
\draw (2,-1) node [] {E: Cylindric PP};
\end{scope}
\end{tikzpicture}
$$

$$
\begin{tikzpicture}[scale=0.6]
\begin{scope}[xshift=0, yshift=0]
\fill [gray!30](2.8000,0.8000)--(2.6000,0.8000)--(2.6000,1.0000)--(2.4000,1.0000)--(2.4000,1.2000)--(2.2000,1.2000)--(2.2000,1.4000)--(2.0000,1.4000)--(2.0000,1.6000)--(1.8000,1.6000)--(1.8000,1.8000)--(1.6000,1.8000)--(1.6000,2.0000)--(1.4000,2.0000)--(1.4000,2.2000)--(1.2000,2.2000)--(1.2000,2.4000)--(1.0000,2.4000)--(1.0000,2.6000)--(0.8000,2.6000)--(0.8000,2.8000)--(0.6000,2.8000)--(0.6000,3.0000)--(0.4000,3.0000)--(0.4000,3.2000)--(0.2000,3.2000)--(0.2000,3.4000)--(0.4000,3.4000)--(0.6000,3.4000)--(0.8000,3.4000)--(1.0000,3.4000)--(1.2000,3.4000)--(1.4000,3.4000)--(1.6000,3.4000)--(1.8000,3.4000)--(2.0000,3.4000)--(2.0000,3.2000)--(2.2000,3.2000)--(2.2000,3.0000)--(2.4000,3.0000)--(2.4000,2.8000)--(2.6000,2.8000)--(2.6000,2.6000)--(2.8000,2.6000)--(2.8000,2.4000)--(3.0000,2.4000)--(3.0000,2.2000)--(3.2000,2.2000)--(3.2000,2.0000)--(3.4000,2.0000)--(3.4000,1.8000)--(3.6000,1.8000)--(3.6000,1.6000)--(3.4000,1.6000)--(3.4000,1.4000)--(3.2000,1.4000)--(3.2000,1.2000)--(3.0000,1.2000)--(3.0000,1.0000)--(2.8000,1.0000)--(2.8000,0.8000);
\draw [gray!10](0.2000,3.2000)--(0.2000,3.4000);
\draw [gray!10](0.4000,3.0000)--(0.4000,3.4000);
\draw [gray!10](0.6000,2.8000)--(0.6000,3.4000);
\draw [gray!10](0.8000,2.6000)--(0.8000,3.4000);
\draw [gray!10](1.0000,2.4000)--(1.0000,3.4000);
\draw [gray!10](1.2000,2.2000)--(1.2000,3.4000);
\draw [gray!10](1.4000,2.0000)--(1.4000,3.4000);
\draw [gray!10](1.6000,1.8000)--(1.6000,3.4000);
\draw [gray!10](1.8000,1.6000)--(1.8000,3.4000);
\draw [gray!10](2.0000,1.4000)--(2.0000,3.4000);
\draw [gray!10](2.2000,1.2000)--(2.2000,3.2000);
\draw [gray!10](2.4000,1.0000)--(2.4000,3.0000);
\draw [gray!10](2.6000,0.8000)--(2.6000,2.8000);
\draw [gray!10](2.8000,0.8000)--(2.8000,2.6000);
\draw [gray!10](3.0000,1.0000)--(3.0000,2.4000);
\draw [gray!10](3.2000,1.2000)--(3.2000,2.2000);
\draw [gray!10](3.4000,1.4000)--(3.4000,2.0000);
\draw [gray!10](2.6000,0.8000)--(2.8000,0.8000);
\draw [gray!10](2.4000,1.0000)--(3.0000,1.0000);
\draw [gray!10](2.2000,1.2000)--(3.2000,1.2000);
\draw [gray!10](2.0000,1.4000)--(3.4000,1.4000);
\draw [gray!10](1.8000,1.6000)--(3.6000,1.6000);
\draw [gray!10](1.6000,1.8000)--(3.6000,1.8000);
\draw [gray!10](1.4000,2.0000)--(3.4000,2.0000);
\draw [gray!10](1.2000,2.2000)--(3.2000,2.2000);
\draw [gray!10](1.0000,2.4000)--(3.0000,2.4000);
\draw [gray!10](0.8000,2.6000)--(2.8000,2.6000);
\draw [gray!10](0.6000,2.8000)--(2.6000,2.8000);
\draw [gray!10](0.4000,3.0000)--(2.4000,3.0000);
\draw [gray!10](0.2000,3.2000)--(2.2000,3.2000);
\draw [black](2.8000,0.8000)--(2.6000,0.8000)--(2.6000,1.0000)--(2.4000,1.0000)--(2.4000,1.2000)--(2.2000,1.2000)--(2.2000,1.4000)--(2.0000,1.4000)--(2.0000,1.6000)--(1.8000,1.6000)--(1.8000,1.8000)--(1.6000,1.8000)--(1.6000,2.0000)--(1.4000,2.0000)--(1.4000,2.2000)--(1.2000,2.2000)--(1.2000,2.4000)--(1.0000,2.4000)--(1.0000,2.6000)--(0.8000,2.6000)--(0.8000,2.8000)--(0.6000,2.8000)--(0.6000,3.0000)--(0.4000,3.0000)--(0.4000,3.2000)--(0.2000,3.2000)--(0.2000,3.4000)--(0.4000,3.4000)--(0.6000,3.4000)--(0.8000,3.4000)--(1.0000,3.4000)--(1.2000,3.4000)--(1.4000,3.4000)--(1.6000,3.4000)--(1.8000,3.4000)--(2.0000,3.4000)--(2.0000,3.2000)--(2.2000,3.2000)--(2.2000,3.0000)--(2.4000,3.0000)--(2.4000,2.8000)--(2.6000,2.8000)--(2.6000,2.6000)--(2.8000,2.6000)--(2.8000,2.4000)--(3.0000,2.4000)--(3.0000,2.2000)--(3.2000,2.2000)--(3.2000,2.0000)--(3.4000,2.0000)--(3.4000,1.8000)--(3.6000,1.8000)--(3.6000,1.6000);
\draw (0.2000,3.4000)--(0.2000,-0.2000);
\draw (0.2000,3.4000)--(4.6000,3.4000);
\draw (2,-1) node [] {F: Doubled Shifted PP};
\end{scope}
\begin{scope}[xshift=175, yshift=0]
\fill [gray!30](1.8000,0.2000)--(1.6000,0.2000)--(1.6000,0.4000)--(1.4000,0.4000)--(1.4000,0.6000)--(1.2000,0.6000)--(1.2000,0.8000)--(1.0000,0.8000)--(1.0000,1.0000)--(0.8000,1.0000)--(0.8000,1.2000)--(0.6000,1.2000)--(0.6000,1.4000)--(0.4000,1.4000)--(0.4000,1.6000)--(0.2000,1.6000)--(0.2000,1.8000)--(0.0000,1.8000)--(0.0000,2.0000)--(0.0000,2.2000)--(0.0000,2.4000)--(0.0000,2.6000)--(0.0000,2.8000)--(0.2000,2.8000)--(0.4000,2.8000)--(0.4000,3.0000)--(0.6000,3.0000)--(0.8000,3.0000)--(1.0000,3.0000)--(1.0000,3.2000)--(1.2000,3.2000)--(1.2000,3.4000)--(1.4000,3.4000)--(1.6000,3.4000)--(1.8000,3.4000)--(2.0000,3.4000)--(2.2000,3.4000)--(2.4000,3.4000)--(2.6000,3.4000)--(2.8000,3.4000)--(3.0000,3.4000)--(3.2000,3.4000)--(3.2000,3.2000)--(3.4000,3.2000)--(3.4000,3.0000)--(3.6000,3.0000)--(3.6000,2.8000)--(3.8000,2.8000)--(3.8000,2.6000)--(4.0000,2.6000)--(4.0000,2.4000)--(3.8000,2.4000)--(3.8000,2.2000)--(3.6000,2.2000)--(3.6000,2.0000)--(3.4000,2.0000)--(3.4000,1.8000)--(3.2000,1.8000)--(3.2000,1.6000)--(3.0000,1.6000)--(3.0000,1.4000)--(2.8000,1.4000)--(2.8000,1.2000)--(2.6000,1.2000)--(2.6000,1.0000)--(2.4000,1.0000)--(2.4000,0.8000)--(2.2000,0.8000)--(2.2000,0.6000)--(2.0000,0.6000)--(2.0000,0.4000)--(1.8000,0.4000)--(1.8000,0.2000);
\draw [gray!10](0.0000,1.8000)--(0.0000,2.8000);
\draw [gray!10](0.2000,1.6000)--(0.2000,2.8000);
\draw [gray!10](0.4000,1.4000)--(0.4000,3.0000);
\draw [gray!10](0.6000,1.2000)--(0.6000,3.0000);
\draw [gray!10](0.8000,1.0000)--(0.8000,3.0000);
\draw [gray!10](1.0000,0.8000)--(1.0000,3.2000);
\draw [gray!10](1.2000,0.6000)--(1.2000,3.4000);
\draw [gray!10](1.4000,0.4000)--(1.4000,3.4000);
\draw [gray!10](1.6000,0.2000)--(1.6000,3.4000);
\draw [gray!10](1.8000,0.2000)--(1.8000,3.4000);
\draw [gray!10](2.0000,0.4000)--(2.0000,3.4000);
\draw [gray!10](2.2000,0.6000)--(2.2000,3.4000);
\draw [gray!10](2.4000,0.8000)--(2.4000,3.4000);
\draw [gray!10](2.6000,1.0000)--(2.6000,3.4000);
\draw [gray!10](2.8000,1.2000)--(2.8000,3.4000);
\draw [gray!10](3.0000,1.4000)--(3.0000,3.4000);
\draw [gray!10](3.2000,1.6000)--(3.2000,3.4000);
\draw [gray!10](3.4000,1.8000)--(3.4000,3.2000);
\draw [gray!10](3.6000,2.0000)--(3.6000,3.0000);
\draw [gray!10](3.8000,2.2000)--(3.8000,2.8000);
\draw [gray!10](1.6000,0.2000)--(1.8000,0.2000);
\draw [gray!10](1.4000,0.4000)--(2.0000,0.4000);
\draw [gray!10](1.2000,0.6000)--(2.2000,0.6000);
\draw [gray!10](1.0000,0.8000)--(2.4000,0.8000);
\draw [gray!10](0.8000,1.0000)--(2.6000,1.0000);
\draw [gray!10](0.6000,1.2000)--(2.8000,1.2000);
\draw [gray!10](0.4000,1.4000)--(3.0000,1.4000);
\draw [gray!10](0.2000,1.6000)--(3.2000,1.6000);
\draw [gray!10](0.0000,1.8000)--(3.4000,1.8000);
\draw [gray!10](0.0000,2.0000)--(3.6000,2.0000);
\draw [gray!10](0.0000,2.2000)--(3.8000,2.2000);
\draw [gray!10](0.0000,2.4000)--(4.0000,2.4000);
\draw [gray!10](0.0000,2.6000)--(4.0000,2.6000);
\draw [gray!10](0.0000,2.8000)--(3.8000,2.8000);
\draw [gray!10](0.4000,3.0000)--(3.6000,3.0000);
\draw [gray!10](1.0000,3.2000)--(3.4000,3.2000);
\draw [black](1.8000,0.2000)--(1.6000,0.2000)--(1.6000,0.4000)--(1.4000,0.4000)--(1.4000,0.6000)--(1.2000,0.6000)--(1.2000,0.8000)--(1.0000,0.8000)--(1.0000,1.0000)--(0.8000,1.0000)--(0.8000,1.2000)--(0.6000,1.2000)--(0.6000,1.4000)--(0.4000,1.4000)--(0.4000,1.6000)--(0.2000,1.6000)--(0.2000,1.8000)--(0.0000,1.8000)--(0.0000,2.0000)--(0.0000,2.2000)--(0.0000,2.4000)--(0.0000,2.6000)--(0.0000,2.8000)--(0.2000,2.8000)--(0.4000,2.8000)--(0.4000,3.0000)--(0.6000,3.0000)--(0.8000,3.0000)--(1.0000,3.0000)--(1.0000,3.2000)--(1.2000,3.2000)--(1.2000,3.4000)--(1.4000,3.4000)--(1.6000,3.4000)--(1.8000,3.4000)--(2.0000,3.4000)--(2.2000,3.4000)--(2.4000,3.4000)--(2.6000,3.4000)--(2.8000,3.4000)--(3.0000,3.4000)--(3.2000,3.4000)--(3.2000,3.2000)--(3.4000,3.2000)--(3.4000,3.0000)--(3.6000,3.0000)--(3.6000,2.8000)--(3.8000,2.8000)--(3.8000,2.6000)--(4.0000,2.6000)--(4.0000,2.4000);
\draw (0.0000,3.4000)--(0.0000,-0.2000);
\draw (0.0000,3.4000)--(4.4000,3.4000);
\draw (2,-1) node [] {G: Skew DSPP};
\end{scope}
\begin{scope}[xshift=350, yshift=0]
\fill [gray!30](1.6000,0.0000)--(1.4000,0.0000)--(1.4000,0.2000)--(1.2000,0.2000)--(1.2000,0.4000)--(1.0000,0.4000)--(1.0000,0.6000)--(0.8000,0.6000)--(0.8000,0.8000)--(0.6000,0.8000)--(0.6000,1.0000)--(0.4000,1.0000)--(0.4000,1.2000)--(0.2000,1.2000)--(0.2000,1.4000)--(0.0000,1.4000)--(0.0000,1.6000)--(0.0000,1.8000)--(0.0000,2.0000)--(0.0000,2.2000)--(0.0000,2.4000)--(0.2000,2.4000)--(0.2000,2.6000)--(0.4000,2.6000)--(0.6000,2.6000)--(0.6000,2.8000)--(0.8000,2.8000)--(0.8000,3.0000)--(0.8000,3.2000)--(1.0000,3.2000)--(1.0000,3.4000)--(1.2000,3.4000)--(1.4000,3.4000)--(1.6000,3.4000)--(1.8000,3.4000)--(2.0000,3.4000)--(2.0000,3.2000)--(2.2000,3.2000)--(2.2000,3.0000)--(2.4000,3.0000)--(2.4000,2.8000)--(2.6000,2.8000)--(2.6000,2.6000)--(2.8000,2.6000)--(2.8000,2.4000)--(3.0000,2.4000)--(3.0000,2.2000)--(3.2000,2.2000)--(3.2000,2.0000)--(3.4000,2.0000)--(3.4000,1.8000)--(3.2000,1.8000)--(3.2000,1.6000)--(3.0000,1.6000)--(3.0000,1.4000)--(2.8000,1.4000)--(2.8000,1.2000)--(2.6000,1.2000)--(2.6000,1.0000)--(2.4000,1.0000)--(2.4000,0.8000)--(2.2000,0.8000)--(2.2000,0.6000)--(2.0000,0.6000)--(2.0000,0.4000)--(1.8000,0.4000)--(1.8000,0.2000)--(1.6000,0.2000)--(1.6000,0.0000);
\draw [gray!10](0.0000,1.4000)--(0.0000,2.4000);
\draw [gray!10](0.2000,1.2000)--(0.2000,2.6000);
\draw [gray!10](0.4000,1.0000)--(0.4000,2.6000);
\draw [gray!10](0.6000,0.8000)--(0.6000,2.8000);
\draw [gray!10](0.8000,0.6000)--(0.8000,3.2000);
\draw [gray!10](1.0000,0.4000)--(1.0000,3.4000);
\draw [gray!10](1.2000,0.2000)--(1.2000,3.4000);
\draw [gray!10](1.4000,0.0000)--(1.4000,3.4000);
\draw [gray!10](1.6000,0.0000)--(1.6000,3.4000);
\draw [gray!10](1.8000,0.2000)--(1.8000,3.4000);
\draw [gray!10](2.0000,0.4000)--(2.0000,3.4000);
\draw [gray!10](2.2000,0.6000)--(2.2000,3.2000);
\draw [gray!10](2.4000,0.8000)--(2.4000,3.0000);
\draw [gray!10](2.6000,1.0000)--(2.6000,2.8000);
\draw [gray!10](2.8000,1.2000)--(2.8000,2.6000);
\draw [gray!10](3.0000,1.4000)--(3.0000,2.4000);
\draw [gray!10](3.2000,1.6000)--(3.2000,2.2000);
\draw [gray!10](1.4000,0.0000)--(1.6000,0.0000);
\draw [gray!10](1.2000,0.2000)--(1.8000,0.2000);
\draw [gray!10](1.0000,0.4000)--(2.0000,0.4000);
\draw [gray!10](0.8000,0.6000)--(2.2000,0.6000);
\draw [gray!10](0.6000,0.8000)--(2.4000,0.8000);
\draw [gray!10](0.4000,1.0000)--(2.6000,1.0000);
\draw [gray!10](0.2000,1.2000)--(2.8000,1.2000);
\draw [gray!10](0.0000,1.4000)--(3.0000,1.4000);
\draw [gray!10](0.0000,1.6000)--(3.2000,1.6000);
\draw [gray!10](0.0000,1.8000)--(3.4000,1.8000);
\draw [gray!10](0.0000,2.0000)--(3.4000,2.0000);
\draw [gray!10](0.0000,2.2000)--(3.2000,2.2000);
\draw [gray!10](0.0000,2.4000)--(3.0000,2.4000);
\draw [gray!10](0.2000,2.6000)--(2.8000,2.6000);
\draw [gray!10](0.6000,2.8000)--(2.6000,2.8000);
\draw [gray!10](0.8000,3.0000)--(2.4000,3.0000);
\draw [gray!10](0.8000,3.2000)--(2.2000,3.2000);
\draw [black](1.6000,0.0000)--(1.4000,0.0000)--(1.4000,0.2000)--(1.2000,0.2000)--(1.2000,0.4000)--(1.0000,0.4000)--(1.0000,0.6000)--(0.8000,0.6000)--(0.8000,0.8000)--(0.6000,0.8000)--(0.6000,1.0000)--(0.4000,1.0000)--(0.4000,1.2000)--(0.2000,1.2000)--(0.2000,1.4000)--(0.0000,1.4000)--(0.0000,1.6000)--(0.0000,1.8000)--(0.0000,2.0000)--(0.0000,2.2000)--(0.0000,2.4000)--(0.2000,2.4000)--(0.2000,2.6000)--(0.4000,2.6000)--(0.6000,2.6000)--(0.6000,2.8000)--(0.8000,2.8000)--(0.8000,3.0000)--(0.8000,3.2000)--(1.0000,3.2000)--(1.0000,3.4000)--(1.2000,3.4000)--(1.4000,3.4000)--(1.6000,3.4000)--(1.8000,3.4000)--(2.0000,3.4000)--(2.0000,3.2000)--(2.2000,3.2000)--(2.2000,3.0000)--(2.4000,3.0000)--(2.4000,2.8000)--(2.6000,2.8000)--(2.6000,2.6000)--(2.8000,2.6000)--(2.8000,2.4000)--(3.0000,2.4000)--(3.0000,2.2000)--(3.2000,2.2000)--(3.2000,2.0000)--(3.4000,2.0000)--(3.4000,1.8000);
\draw (0.0000,3.4000)--(0.0000,-0.2000);
\draw (0.0000,3.4000)--(3.8000,3.4000);
\draw [dashed] (0.0000,3.4000)--(3.6000,-0.2000);
\draw (2.3,  2.4) node [] {$\lambda$};
\draw (1.1,  1.2) node [] {$\lambda$};
\draw (2,-1) node [] {H: Sym Cylindric PP};
\end{scope}
\end{tikzpicture}
$$
\centerline{Fig. 1. Various kinds of defective plane partitions.}

\medskip

The Schur process was first introduced by Okounkov and Reshetikhin \cite{Okounkov2001,Okounkov2003} in $2001$. 
Later, they used the Schur process 
to derive the cyclic symmetry of the topological vertex
by considering a 
certain type of plane partitions \cite{Okoun2006}. 
This result was further developed by Iqbal et al.,
 for providing a short proof of the Nekrasov-Okounkov formula
\cite{Han2010,Iqbal2012,NekrasovOkounkov2006,Westbury2006}.
Borodin \cite{Borodin2007,Borodin2011} used the Schur process to derive the generating function for cylindric partitions,  introduced by Gessel and Krattenthaler \cite{GesselKratt1997}. The
Macdonald process, which is a $(q,t)$-generalization of the Schur process, was first introduced by Vuletic \cite{Vuletic2009},
and further developed by Corteel, Savelief, Vuleti\'c and Langer
 \cite{CorteelSaveliefVuletic,  Langer2013A, Langer2013B, Vuletic2007}
 in the study of 
weighted cylindric partitions 
and plane overpartitions.
Finally,  a survey of Macdonald processes was published 
by Borodin and Corwin \cite{Borodin2014}.

\smallskip

%
% Our approach is the second one, based on a summation formula of skew Schur functions. 
The Schur process approach is shown to be a powerful tool in the study of
various kinds of defective plane partitions.
In fact,
the above generating functions for defective plane partitions (A-E)
are specializations of
two general summation formulas for Schur processes, namely, 
the {\it open summation formula} \eqref{eq:MainOpen}
and the {\it cylindric summation formula}
\eqref{eq:MainCylinder}.
Formulas \eqref{eq:MainOpen} and \eqref{eq:MainCylinder}
have been developed by Okounkov, Reshetikhin, Borodin, Corteel, Savelief, Vuleti\'c and Langer \cite{Borodin2007,Borodin2011,CorteelSaveliefVuletic,  Langer2013A, Langer2013B,Okounkov2001,Okounkov2003, Vuletic2007,Vuletic2009}.
For convenience, they are also reproduced in Theorem  
\ref{th:MainZ}.

In the present paper we establish a new	summation formula for Schur processes,
called the {\it complete summation formula} \eqref{eq:MainComplete}. 
As an application, we obtain the generating functions  for
the skew doubled shifted plane partitions 
and the symmetric cylindric partitions (see Fig. 1 (F)/(G)/(H)).

\medskip

Let us reproduce some classical formulas in this introduction (see, e.g., \cite{Macdonald1995,Sagan1982,Sagan1993,StanleyEC2}).
The generating functions for ordinary plane partitions,
shifted plane partitions and symmetric plane partitions are the following
 respectively:
\begin{align} 
	\sum_{\omega \in \PP{PP}} z^{|\omega|} 
	&= 
	\prod_{i=1}^{\infty}\prod_{j=1}^{\infty}\frac{1}{1-z^{i+j-1}}=\prod_{k=1}^\infty (1-z^k)^{-k}; \label{eq:gfPP} \\
	\sum_{\omega \in \PP{ShiftPP}} z^{|\omega|} 
	&= 
\prod_{k=1}^\infty \frac{1}{1-z^{k}}
	\prod_{1\leq i< j \leq \infty} \frac{1}{1-z^{i+j}}; \label{eq:gfSPP}
\\
\sum_{\omega \in \PP{SPP}} z^{|\omega|} 
&=
	\prod_{\substack{k=1}}^\infty \frac{1}{1-z^{2k-1}}\prod_{\substack{1\leq i<j\leq \infty}} \frac{1}{1-z^{2(i+j-1)}}.\label{eq:gfSyPP}
\end{align}

As a byproduct of our complete summation formula \eqref{eq:MainComplete} we can further derive the following generating function
for doubled shifted plane partitions of width $m$ (see Fig. 1 (F) and Section 
\ref{sec:SkRPP} for the definition).
\begin{thm}\label{th:gfRPP}
	Let $\PP{DSPP}_m$ be the set of all doubled shifted plane partitions $\omega$ of width $m$ (i.e., skew doubled shifted plane partitions with profile $\delta=(-1)^{m-1}$).
Then
\begin{equation}\label{eq:gfRPP}
\sum_{\omega \in \PP{DSPP}_m}  z^{|\omega|}
=  
	\prod_{k= 1}^{\infty} \frac{1}{1-z^{k}} \times
	\prod_{k= 0}^{\infty}  \,
	\prod_{1\leq i<j\leq m-1} \frac{1}{1-z^{2mk+i+j}}.
\end{equation}
\end{thm}

Inspired by the works of Dewar, Murty and Kot\v{e}\v{s}ovec \cite{DewarMurty2013, Kotesovec2015}, we establish some useful theorems for
asymptotic formulas in \cite{HanXiong2017} (see Section \ref{sec:Asym}).
Furthermore, 
the following asymptotic formula for the number of doubled shifted plane partitions can also be obtained.
\begin{thm}\label{th:AsymRPP}
	Let $\PP{DSPP}_m(n)$ be the number of 
	doubled shifted plane partitions $\omega$ of width $m$ and size $n$. Then, 
	\begin{equation}\label{eq:AsymRPP}
	\PP{DSPP}_m(n) \ \sim \  
	C(m) \times  \frac{1}{n} 
	\exp\Bigl(\pi\sqrt{\frac{(m^2+m+2)n}{6m}} \, \Bigr),
	\end{equation}
	where $C(m)$ is a constant independent of $n$ given by the following expression:
$$
C(m)=
	\left(\prod_{i=1}^{m-2}\,\prod_{j=i+1}^{m-i-1} {{\sin\left(\frac{i+j}{2m}\pi\right)}} \right)^{-1}
	\frac{\sqrt{m^2+m+2}}{2^{(m^2-3m+14)/4}  \sqrt{3m} } .
$$
\end{thm}

The proofs of Theorems \ref{th:gfRPP} and \ref{th:AsymRPP} are given in Section \ref{sec:gfRPP}.
For example, for $m=3$ in Theorems \ref{th:gfRPP} and \ref{th:AsymRPP}, 
 the generating function and asymptotic formula for 
doubled shifted plane partitions of width $m=3$ (see Fig. 2, case DSPPa) are
\begin{align} 
	\sum_{\omega \in \PP{DSPP}_{3}} z^{|\omega|} 
	&\ =\  \prod_{k\geq 1} \frac{1}{(1-z^{k})(1-z^{6k-3})};\\
\PP{DSPP}_3(n) 
	&\ \sim	 \ 
	\frac{\sqrt{7}}{24}\,\frac {\exp(\pi \frac{\sqrt {7n}}{3} )}{ n}.\label{eq:asym3a}
\end{align}

In fact, our Theorems \ref{th:gfRPP} and \ref{th:AsymRPP}
can be extended to {\it skew doubled shifted plane partitions} 
(see Sections \ref{sec:SkRPP} and \ref{sec:gfRPP}).
The asymptotic formulas for two other skew doubled shifted plane partitions,
together with some ordinary plane partitions (\PP{PP}), cylindric partitions (\PP{CP}) and symmetric cylindric partitions  (\PP{SCP}) are also reproduced next. The proofs of those asymptotic formulas can be found in
Sections \ref{sec:gfRPP} and \ref{sec:gfSyCPP} for \PP{DSPP} and \PP{SCP}
respectively, 
and in \cite{HanXiong2017} for \PP{PP} and \PP{CP}.

\begin{align*}
\begin{tikzpicture}[scale=0.6]
\begin{scope}[xshift=0, yshift=0]
\fill [gray!30](1.8000,0.2000)--(1.6000,0.2000)--(1.4000,0.2000)--(1.2000,0.2000)--(1.0000,0.2000)--(0.8000,0.2000)--(0.6000,0.2000)--(0.4000,0.2000)--(0.2000,0.2000)--(0.0000,0.2000)--(-0.2000,0.2000)--(-0.4000,0.2000)--(-0.6000,0.2000)--(-0.8000,0.2000)--(-0.8000,0.4000)--(-0.8000,0.6000)--(-0.8000,0.8000)--(-0.6000,0.8000)--(-0.4000,0.8000)--(-0.2000,0.8000)--(0.0000,0.8000)--(0.2000,0.8000)--(0.4000,0.8000)--(0.6000,0.8000)--(0.8000,0.8000)--(1.0000,0.8000)--(1.2000,0.8000)--(1.4000,0.8000)--(1.6000,0.8000)--(1.8000,0.8000)--(1.8000,0.6000)--(1.8000,0.4000)--(1.8000,0.2000);
\draw [gray!10](-0.8000,0.2000)--(-0.8000,0.8000);
\draw [gray!10](-0.6000,0.2000)--(-0.6000,0.8000);
\draw [gray!10](-0.4000,0.2000)--(-0.4000,0.8000);
\draw [gray!10](-0.2000,0.2000)--(-0.2000,0.8000);
\draw [gray!10](0.0000,0.2000)--(0.0000,0.8000);
\draw [gray!10](0.2000,0.2000)--(0.2000,0.8000);
\draw [gray!10](0.4000,0.2000)--(0.4000,0.8000);
\draw [gray!10](0.6000,0.2000)--(0.6000,0.8000);
\draw [gray!10](0.8000,0.2000)--(0.8000,0.8000);
\draw [gray!10](1.0000,0.2000)--(1.0000,0.8000);
\draw [gray!10](1.2000,0.2000)--(1.2000,0.8000);
\draw [gray!10](1.4000,0.2000)--(1.4000,0.8000);
\draw [gray!10](1.6000,0.2000)--(1.6000,0.8000);
\draw [gray!10](-0.8000,0.2000)--(1.8000,0.2000);
\draw [gray!10](-0.8000,0.4000)--(1.8000,0.4000);
\draw [gray!10](-0.8000,0.6000)--(1.8000,0.6000);
\draw [black](1.8000,0.2000)--(1.6000,0.2000)--(1.4000,0.2000)--(1.2000,0.2000)--(1.0000,0.2000)--(0.8000,0.2000)--(0.6000,0.2000)--(0.4000,0.2000)--(0.2000,0.2000)--(0.0000,0.2000)--(-0.2000,0.2000)--(-0.4000,0.2000)--(-0.6000,0.2000)--(-0.8000,0.2000)--(-0.8000,0.4000)--(-0.8000,0.6000)--(-0.8000,0.8000)--(-0.6000,0.8000)--(-0.4000,0.8000)--(-0.2000,0.8000)--(0.0000,0.8000)--(0.2000,0.8000)--(0.4000,0.8000)--(0.6000,0.8000)--(0.8000,0.8000)--(1.0000,0.8000)--(1.2000,0.8000)--(1.4000,0.8000)--(1.6000,0.8000)--(1.8000,0.8000);
\draw (-0.8000,0.8000)--(-0.8000,-0.2000);
\draw (-0.8000,0.8000)--(2.2000,0.8000);
\draw (0.8,-1) node [] {PPa $\sim \frac{1.93}{n^3} e^{4.44 \sqrt{n}}$};
\end{scope}
\begin{scope}[xshift=180, yshift=0]
\fill [gray!30](1.8000,0.2000)--(1.6000,0.2000)--(1.4000,0.2000)--(1.2000,0.2000)--(1.0000,0.2000)--(0.8000,0.2000)--(0.6000,0.2000)--(0.4000,0.2000)--(0.2000,0.2000)--(0.0000,0.2000)--(-0.2000,0.2000)--(-0.4000,0.2000)--(-0.6000,0.2000)--(-0.8000,0.2000)--(-0.8000,0.4000)--(-0.8000,0.6000)--(-0.6000,0.6000)--(-0.6000,0.8000)--(-0.4000,0.8000)--(-0.2000,0.8000)--(0.0000,0.8000)--(0.2000,0.8000)--(0.4000,0.8000)--(0.6000,0.8000)--(0.8000,0.8000)--(1.0000,0.8000)--(1.2000,0.8000)--(1.4000,0.8000)--(1.6000,0.8000)--(1.8000,0.8000)--(1.8000,0.6000)--(1.8000,0.4000)--(1.8000,0.2000);
\draw [gray!10](-0.8000,0.2000)--(-0.8000,0.6000);
\draw [gray!10](-0.6000,0.2000)--(-0.6000,0.8000);
\draw [gray!10](-0.4000,0.2000)--(-0.4000,0.8000);
\draw [gray!10](-0.2000,0.2000)--(-0.2000,0.8000);
\draw [gray!10](0.0000,0.2000)--(0.0000,0.8000);
\draw [gray!10](0.2000,0.2000)--(0.2000,0.8000);
\draw [gray!10](0.4000,0.2000)--(0.4000,0.8000);
\draw [gray!10](0.6000,0.2000)--(0.6000,0.8000);
\draw [gray!10](0.8000,0.2000)--(0.8000,0.8000);
\draw [gray!10](1.0000,0.2000)--(1.0000,0.8000);
\draw [gray!10](1.2000,0.2000)--(1.2000,0.8000);
\draw [gray!10](1.4000,0.2000)--(1.4000,0.8000);
\draw [gray!10](1.6000,0.2000)--(1.6000,0.8000);
\draw [gray!10](-0.8000,0.2000)--(1.8000,0.2000);
\draw [gray!10](-0.8000,0.4000)--(1.8000,0.4000);
\draw [gray!10](-0.8000,0.6000)--(1.8000,0.6000);
\draw [black](1.8000,0.2000)--(1.6000,0.2000)--(1.4000,0.2000)--(1.2000,0.2000)--(1.0000,0.2000)--(0.8000,0.2000)--(0.6000,0.2000)--(0.4000,0.2000)--(0.2000,0.2000)--(0.0000,0.2000)--(-0.2000,0.2000)--(-0.4000,0.2000)--(-0.6000,0.2000)--(-0.8000,0.2000)--(-0.8000,0.4000)--(-0.8000,0.6000)--(-0.6000,0.6000)--(-0.6000,0.8000)--(-0.4000,0.8000)--(-0.2000,0.8000)--(0.0000,0.8000)--(0.2000,0.8000)--(0.4000,0.8000)--(0.6000,0.8000)--(0.8000,0.8000)--(1.0000,0.8000)--(1.2000,0.8000)--(1.4000,0.8000)--(1.6000,0.8000)--(1.8000,0.8000);
\draw (-0.8000,0.8000)--(-0.8000,-0.2000);
\draw (-0.8000,0.8000)--(2.2000,0.8000);
\draw (0.8,-1) node [] {PPb $\sim \frac{5.81}{n^3} e^{4.44 \sqrt{n}}$};
\end{scope}
%\begin{scope}[xshift=280, yshift=0]
%\input Fig_3C.tex
%\draw (1,-1) node [] {F?: Ribbon PP};
%\end{scope}
\begin{scope}[xshift=360, yshift=0]
\fill [gray!30](1.8000,0.2000)--(1.6000,0.2000)--(1.4000,0.2000)--(1.2000,0.2000)--(1.0000,0.2000)--(0.8000,0.2000)--(0.6000,0.2000)--(0.4000,0.2000)--(0.2000,0.2000)--(0.0000,0.2000)--(-0.2000,0.2000)--(-0.4000,0.2000)--(-0.6000,0.2000)--(-0.8000,0.2000)--(-0.8000,0.4000)--(-0.8000,0.6000)--(-0.6000,0.6000)--(-0.4000,0.6000)--(-0.4000,0.8000)--(-0.2000,0.8000)--(0.0000,0.8000)--(0.2000,0.8000)--(0.4000,0.8000)--(0.6000,0.8000)--(0.8000,0.8000)--(1.0000,0.8000)--(1.2000,0.8000)--(1.4000,0.8000)--(1.6000,0.8000)--(1.8000,0.8000)--(1.8000,0.6000)--(1.8000,0.4000)--(1.8000,0.2000);
\draw [gray!10](-0.8000,0.2000)--(-0.8000,0.6000);
\draw [gray!10](-0.6000,0.2000)--(-0.6000,0.6000);
\draw [gray!10](-0.4000,0.2000)--(-0.4000,0.8000);
\draw [gray!10](-0.2000,0.2000)--(-0.2000,0.8000);
\draw [gray!10](0.0000,0.2000)--(0.0000,0.8000);
\draw [gray!10](0.2000,0.2000)--(0.2000,0.8000);
\draw [gray!10](0.4000,0.2000)--(0.4000,0.8000);
\draw [gray!10](0.6000,0.2000)--(0.6000,0.8000);
\draw [gray!10](0.8000,0.2000)--(0.8000,0.8000);
\draw [gray!10](1.0000,0.2000)--(1.0000,0.8000);
\draw [gray!10](1.2000,0.2000)--(1.2000,0.8000);
\draw [gray!10](1.4000,0.2000)--(1.4000,0.8000);
\draw [gray!10](1.6000,0.2000)--(1.6000,0.8000);
\draw [gray!10](-0.8000,0.2000)--(1.8000,0.2000);
\draw [gray!10](-0.8000,0.4000)--(1.8000,0.4000);
\draw [gray!10](-0.8000,0.6000)--(1.8000,0.6000);
\draw [black](1.8000,0.2000)--(1.6000,0.2000)--(1.4000,0.2000)--(1.2000,0.2000)--(1.0000,0.2000)--(0.8000,0.2000)--(0.6000,0.2000)--(0.4000,0.2000)--(0.2000,0.2000)--(0.0000,0.2000)--(-0.2000,0.2000)--(-0.4000,0.2000)--(-0.6000,0.2000)--(-0.8000,0.2000)--(-0.8000,0.4000)--(-0.8000,0.6000)--(-0.6000,0.6000)--(-0.4000,0.6000)--(-0.4000,0.8000)--(-0.2000,0.8000)--(0.0000,0.8000)--(0.2000,0.8000)--(0.4000,0.8000)--(0.6000,0.8000)--(0.8000,0.8000)--(1.0000,0.8000)--(1.2000,0.8000)--(1.4000,0.8000)--(1.6000,0.8000)--(1.8000,0.8000);
\draw (-0.8000,0.8000)--(-0.8000,-0.2000);
\draw (-0.8000,0.8000)--(2.2000,0.8000);
\draw (0.8,-1) node [] {PPc $\sim \frac{11.62}{n^3} e^{4.44 \sqrt{n}}$};
\end{scope}
\end{tikzpicture}
\\
\begin{tikzpicture}[scale=0.6]
\begin{scope}[xshift=0, yshift=0]
\fill [gray!30](3.2000,0.2000)--(3.2000,0.4000)--(3.0000,0.4000)--(3.0000,0.6000)--(2.8000,0.6000)--(2.8000,0.8000)--(2.6000,0.8000)--(2.6000,1.0000)--(2.4000,1.0000)--(2.4000,1.2000)--(2.2000,1.2000)--(2.2000,1.4000)--(2.0000,1.4000)--(2.0000,1.6000)--(1.8000,1.6000)--(1.8000,1.8000)--(1.6000,1.8000)--(1.6000,2.0000)--(1.4000,2.0000)--(1.4000,2.2000)--(1.2000,2.2000)--(1.2000,2.4000)--(1.0000,2.4000)--(1.0000,2.6000)--(0.8000,2.6000)--(0.8000,2.8000)--(0.8000,3.0000)--(1.0000,3.0000)--(1.2000,3.0000)--(1.4000,3.0000)--(1.6000,3.0000)--(1.6000,2.8000)--(1.8000,2.8000)--(1.8000,2.6000)--(2.0000,2.6000)--(2.0000,2.4000)--(2.2000,2.4000)--(2.2000,2.2000)--(2.4000,2.2000)--(2.4000,2.0000)--(2.6000,2.0000)--(2.6000,1.8000)--(2.8000,1.8000)--(2.8000,1.6000)--(3.0000,1.6000)--(3.0000,1.4000)--(3.2000,1.4000)--(3.2000,1.2000)--(3.4000,1.2000)--(3.4000,1.0000)--(3.6000,1.0000)--(3.6000,0.8000)--(3.8000,0.8000)--(3.8000,0.6000)--(3.6000,0.6000)--(3.6000,0.4000)--(3.4000,0.4000)--(3.4000,0.2000)--(3.2000,0.2000);
\draw [gray!10](0.8000,2.6000)--(0.8000,3.0000);
\draw [gray!10](1.0000,2.4000)--(1.0000,3.0000);
\draw [gray!10](1.2000,2.2000)--(1.2000,3.0000);
\draw [gray!10](1.4000,2.0000)--(1.4000,3.0000);
\draw [gray!10](1.6000,1.8000)--(1.6000,3.0000);
\draw [gray!10](1.8000,1.6000)--(1.8000,2.8000);
\draw [gray!10](2.0000,1.4000)--(2.0000,2.6000);
\draw [gray!10](2.2000,1.2000)--(2.2000,2.4000);
\draw [gray!10](2.4000,1.0000)--(2.4000,2.2000);
\draw [gray!10](2.6000,0.8000)--(2.6000,2.0000);
\draw [gray!10](2.8000,0.6000)--(2.8000,1.8000);
\draw [gray!10](3.0000,0.4000)--(3.0000,1.6000);
\draw [gray!10](3.2000,0.2000)--(3.2000,1.4000);
\draw [gray!10](3.4000,0.2000)--(3.4000,1.2000);
\draw [gray!10](3.6000,0.4000)--(3.6000,1.0000);
\draw [gray!10](3.2000,0.2000)--(3.4000,0.2000);
\draw [gray!10](3.0000,0.4000)--(3.6000,0.4000);
\draw [gray!10](2.8000,0.6000)--(3.8000,0.6000);
\draw [gray!10](2.6000,0.8000)--(3.8000,0.8000);
\draw [gray!10](2.4000,1.0000)--(3.6000,1.0000);
\draw [gray!10](2.2000,1.2000)--(3.4000,1.2000);
\draw [gray!10](2.0000,1.4000)--(3.2000,1.4000);
\draw [gray!10](1.8000,1.6000)--(3.0000,1.6000);
\draw [gray!10](1.6000,1.8000)--(2.8000,1.8000);
\draw [gray!10](1.4000,2.0000)--(2.6000,2.0000);
\draw [gray!10](1.2000,2.2000)--(2.4000,2.2000);
\draw [gray!10](1.0000,2.4000)--(2.2000,2.4000);
\draw [gray!10](0.8000,2.6000)--(2.0000,2.6000);
\draw [gray!10](0.8000,2.8000)--(1.8000,2.8000);
\draw [black](3.2000,0.2000)--(3.2000,0.4000)--(3.0000,0.4000)--(3.0000,0.6000)--(2.8000,0.6000)--(2.8000,0.8000)--(2.6000,0.8000)--(2.6000,1.0000)--(2.4000,1.0000)--(2.4000,1.2000)--(2.2000,1.2000)--(2.2000,1.4000)--(2.0000,1.4000)--(2.0000,1.6000)--(1.8000,1.6000)--(1.8000,1.8000)--(1.6000,1.8000)--(1.6000,2.0000)--(1.4000,2.0000)--(1.4000,2.2000)--(1.2000,2.2000)--(1.2000,2.4000)--(1.0000,2.4000)--(1.0000,2.6000)--(0.8000,2.6000)--(0.8000,2.8000)--(0.8000,3.0000)--(1.0000,3.0000)--(1.2000,3.0000)--(1.4000,3.0000)--(1.6000,3.0000)--(1.6000,2.8000)--(1.8000,2.8000)--(1.8000,2.6000)--(2.0000,2.6000)--(2.0000,2.4000)--(2.2000,2.4000)--(2.2000,2.2000)--(2.4000,2.2000)--(2.4000,2.0000)--(2.6000,2.0000)--(2.6000,1.8000)--(2.8000,1.8000)--(2.8000,1.6000)--(3.0000,1.6000)--(3.0000,1.4000)--(3.2000,1.4000)--(3.2000,1.2000)--(3.4000,1.2000)--(3.4000,1.0000)--(3.6000,1.0000)--(3.6000,0.8000)--(3.8000,0.8000);
\draw (0.8000,3.0000)--(0.8000,-0.2000);
\draw (0.8000,3.0000)--(4.2000,3.0000);
\draw (3.2,  2.0) node [] {$\lambda$};
\draw (2.0,  0.8) node [] {$\lambda$};
\draw (2.5,-1) node [] {CPa $\sim \frac{0.144}n e^{2.56 \sqrt{n}}$};
\end{scope}
\begin{scope}[xshift=180, yshift=0]
\fill [gray!30](3.2000,0.2000)--(3.2000,0.4000)--(3.0000,0.4000)--(3.0000,0.6000)--(2.8000,0.6000)--(2.8000,0.8000)--(2.6000,0.8000)--(2.6000,1.0000)--(2.4000,1.0000)--(2.4000,1.2000)--(2.2000,1.2000)--(2.2000,1.4000)--(2.0000,1.4000)--(2.0000,1.6000)--(1.8000,1.6000)--(1.8000,1.8000)--(1.6000,1.8000)--(1.6000,2.0000)--(1.4000,2.0000)--(1.4000,2.2000)--(1.2000,2.2000)--(1.2000,2.4000)--(1.0000,2.4000)--(1.0000,2.6000)--(1.0000,2.8000)--(1.2000,2.8000)--(1.2000,3.0000)--(1.4000,3.0000)--(1.6000,3.0000)--(1.6000,2.8000)--(1.8000,2.8000)--(1.8000,2.6000)--(2.0000,2.6000)--(2.0000,2.4000)--(2.2000,2.4000)--(2.2000,2.2000)--(2.4000,2.2000)--(2.4000,2.0000)--(2.6000,2.0000)--(2.6000,1.8000)--(2.8000,1.8000)--(2.8000,1.6000)--(3.0000,1.6000)--(3.0000,1.4000)--(3.2000,1.4000)--(3.2000,1.2000)--(3.4000,1.2000)--(3.4000,1.0000)--(3.6000,1.0000)--(3.6000,0.8000)--(3.8000,0.8000)--(3.8000,0.6000)--(3.6000,0.6000)--(3.6000,0.4000)--(3.4000,0.4000)--(3.4000,0.2000)--(3.2000,0.2000);
\draw [gray!10](1.0000,2.4000)--(1.0000,2.8000);
\draw [gray!10](1.2000,2.2000)--(1.2000,3.0000);
\draw [gray!10](1.4000,2.0000)--(1.4000,3.0000);
\draw [gray!10](1.6000,1.8000)--(1.6000,3.0000);
\draw [gray!10](1.8000,1.6000)--(1.8000,2.8000);
\draw [gray!10](2.0000,1.4000)--(2.0000,2.6000);
\draw [gray!10](2.2000,1.2000)--(2.2000,2.4000);
\draw [gray!10](2.4000,1.0000)--(2.4000,2.2000);
\draw [gray!10](2.6000,0.8000)--(2.6000,2.0000);
\draw [gray!10](2.8000,0.6000)--(2.8000,1.8000);
\draw [gray!10](3.0000,0.4000)--(3.0000,1.6000);
\draw [gray!10](3.2000,0.2000)--(3.2000,1.4000);
\draw [gray!10](3.4000,0.2000)--(3.4000,1.2000);
\draw [gray!10](3.6000,0.4000)--(3.6000,1.0000);
\draw [gray!10](3.2000,0.2000)--(3.4000,0.2000);
\draw [gray!10](3.0000,0.4000)--(3.6000,0.4000);
\draw [gray!10](2.8000,0.6000)--(3.8000,0.6000);
\draw [gray!10](2.6000,0.8000)--(3.8000,0.8000);
\draw [gray!10](2.4000,1.0000)--(3.6000,1.0000);
\draw [gray!10](2.2000,1.2000)--(3.4000,1.2000);
\draw [gray!10](2.0000,1.4000)--(3.2000,1.4000);
\draw [gray!10](1.8000,1.6000)--(3.0000,1.6000);
\draw [gray!10](1.6000,1.8000)--(2.8000,1.8000);
\draw [gray!10](1.4000,2.0000)--(2.6000,2.0000);
\draw [gray!10](1.2000,2.2000)--(2.4000,2.2000);
\draw [gray!10](1.0000,2.4000)--(2.2000,2.4000);
\draw [gray!10](1.0000,2.6000)--(2.0000,2.6000);
\draw [gray!10](1.0000,2.8000)--(1.8000,2.8000);
\draw [black](3.2000,0.2000)--(3.2000,0.4000)--(3.0000,0.4000)--(3.0000,0.6000)--(2.8000,0.6000)--(2.8000,0.8000)--(2.6000,0.8000)--(2.6000,1.0000)--(2.4000,1.0000)--(2.4000,1.2000)--(2.2000,1.2000)--(2.2000,1.4000)--(2.0000,1.4000)--(2.0000,1.6000)--(1.8000,1.6000)--(1.8000,1.8000)--(1.6000,1.8000)--(1.6000,2.0000)--(1.4000,2.0000)--(1.4000,2.2000)--(1.2000,2.2000)--(1.2000,2.4000)--(1.0000,2.4000)--(1.0000,2.6000)--(1.0000,2.8000)--(1.2000,2.8000)--(1.2000,3.0000)--(1.4000,3.0000)--(1.6000,3.0000)--(1.6000,2.8000)--(1.8000,2.8000)--(1.8000,2.6000)--(2.0000,2.6000)--(2.0000,2.4000)--(2.2000,2.4000)--(2.2000,2.2000)--(2.4000,2.2000)--(2.4000,2.0000)--(2.6000,2.0000)--(2.6000,1.8000)--(2.8000,1.8000)--(2.8000,1.6000)--(3.0000,1.6000)--(3.0000,1.4000)--(3.2000,1.4000)--(3.2000,1.2000)--(3.4000,1.2000)--(3.4000,1.0000)--(3.6000,1.0000)--(3.6000,0.8000)--(3.8000,0.8000);
\draw (1.0000,3.0000)--(1.0000,-0.2000);
\draw (1.0000,3.0000)--(4.2000,3.0000);
\draw (3.2,  2.0) node [] {$\lambda$};
\draw (2.0,  0.8) node [] {$\lambda$};
\draw (2.5,-1) node [] {CPb $\sim \frac{0.161}n e^{2.86 \sqrt{n}}$};
\end{scope}
\begin{scope}[xshift=360, yshift=0]
\fill [gray!30](3.2000,0.2000)--(3.2000,0.4000)--(3.0000,0.4000)--(3.0000,0.6000)--(2.8000,0.6000)--(2.8000,0.8000)--(2.6000,0.8000)--(2.6000,1.0000)--(2.4000,1.0000)--(2.4000,1.2000)--(2.2000,1.2000)--(2.2000,1.4000)--(2.0000,1.4000)--(2.0000,1.6000)--(1.8000,1.6000)--(1.8000,1.8000)--(1.6000,1.8000)--(1.6000,2.0000)--(1.4000,2.0000)--(1.4000,2.2000)--(1.2000,2.2000)--(1.2000,2.4000)--(1.0000,2.4000)--(1.0000,2.6000)--(1.0000,2.8000)--(1.2000,2.8000)--(1.4000,2.8000)--(1.4000,3.0000)--(1.6000,3.0000)--(1.6000,2.8000)--(1.8000,2.8000)--(1.8000,2.6000)--(2.0000,2.6000)--(2.0000,2.4000)--(2.2000,2.4000)--(2.2000,2.2000)--(2.4000,2.2000)--(2.4000,2.0000)--(2.6000,2.0000)--(2.6000,1.8000)--(2.8000,1.8000)--(2.8000,1.6000)--(3.0000,1.6000)--(3.0000,1.4000)--(3.2000,1.4000)--(3.2000,1.2000)--(3.4000,1.2000)--(3.4000,1.0000)--(3.6000,1.0000)--(3.6000,0.8000)--(3.8000,0.8000)--(3.8000,0.6000)--(3.6000,0.6000)--(3.6000,0.4000)--(3.4000,0.4000)--(3.4000,0.2000)--(3.2000,0.2000);
\draw [gray!10](1.0000,2.4000)--(1.0000,2.8000);
\draw [gray!10](1.2000,2.2000)--(1.2000,2.8000);
\draw [gray!10](1.4000,2.0000)--(1.4000,3.0000);
\draw [gray!10](1.6000,1.8000)--(1.6000,3.0000);
\draw [gray!10](1.8000,1.6000)--(1.8000,2.8000);
\draw [gray!10](2.0000,1.4000)--(2.0000,2.6000);
\draw [gray!10](2.2000,1.2000)--(2.2000,2.4000);
\draw [gray!10](2.4000,1.0000)--(2.4000,2.2000);
\draw [gray!10](2.6000,0.8000)--(2.6000,2.0000);
\draw [gray!10](2.8000,0.6000)--(2.8000,1.8000);
\draw [gray!10](3.0000,0.4000)--(3.0000,1.6000);
\draw [gray!10](3.2000,0.2000)--(3.2000,1.4000);
\draw [gray!10](3.4000,0.2000)--(3.4000,1.2000);
\draw [gray!10](3.6000,0.4000)--(3.6000,1.0000);
\draw [gray!10](3.2000,0.2000)--(3.4000,0.2000);
\draw [gray!10](3.0000,0.4000)--(3.6000,0.4000);
\draw [gray!10](2.8000,0.6000)--(3.8000,0.6000);
\draw [gray!10](2.6000,0.8000)--(3.8000,0.8000);
\draw [gray!10](2.4000,1.0000)--(3.6000,1.0000);
\draw [gray!10](2.2000,1.2000)--(3.4000,1.2000);
\draw [gray!10](2.0000,1.4000)--(3.2000,1.4000);
\draw [gray!10](1.8000,1.6000)--(3.0000,1.6000);
\draw [gray!10](1.6000,1.8000)--(2.8000,1.8000);
\draw [gray!10](1.4000,2.0000)--(2.6000,2.0000);
\draw [gray!10](1.2000,2.2000)--(2.4000,2.2000);
\draw [gray!10](1.0000,2.4000)--(2.2000,2.4000);
\draw [gray!10](1.0000,2.6000)--(2.0000,2.6000);
\draw [gray!10](1.0000,2.8000)--(1.8000,2.8000);
\draw [black](3.2000,0.2000)--(3.2000,0.4000)--(3.0000,0.4000)--(3.0000,0.6000)--(2.8000,0.6000)--(2.8000,0.8000)--(2.6000,0.8000)--(2.6000,1.0000)--(2.4000,1.0000)--(2.4000,1.2000)--(2.2000,1.2000)--(2.2000,1.4000)--(2.0000,1.4000)--(2.0000,1.6000)--(1.8000,1.6000)--(1.8000,1.8000)--(1.6000,1.8000)--(1.6000,2.0000)--(1.4000,2.0000)--(1.4000,2.2000)--(1.2000,2.2000)--(1.2000,2.4000)--(1.0000,2.4000)--(1.0000,2.6000)--(1.0000,2.8000)--(1.2000,2.8000)--(1.4000,2.8000)--(1.4000,3.0000)--(1.6000,3.0000)--(1.6000,2.8000)--(1.8000,2.8000)--(1.8000,2.6000)--(2.0000,2.6000)--(2.0000,2.4000)--(2.2000,2.4000)--(2.2000,2.2000)--(2.4000,2.2000)--(2.4000,2.0000)--(2.6000,2.0000)--(2.6000,1.8000)--(2.8000,1.8000)--(2.8000,1.6000)--(3.0000,1.6000)--(3.0000,1.4000)--(3.2000,1.4000)--(3.2000,1.2000)--(3.4000,1.2000)--(3.4000,1.0000)--(3.6000,1.0000)--(3.6000,0.8000)--(3.8000,0.8000);
\draw (1.0000,3.0000)--(1.0000,-0.2000);
\draw (1.0000,3.0000)--(4.2000,3.0000);
\draw (3.2,  2.0) node [] {$\lambda$};
\draw (2.0,  0.8) node [] {$\lambda$};
\draw (2.5,-1) node [] {CPc $\sim \frac{0.114}n e^{2.86 \sqrt{n}}$};
\end{scope}
\end{tikzpicture}
\\
\begin{tikzpicture}[scale=0.6]
\begin{scope}[xshift=0, yshift=0]
\fill [gray!30](3.2000,0.2000)--(3.2000,0.4000)--(3.0000,0.4000)--(3.0000,0.6000)--(2.8000,0.6000)--(2.8000,0.8000)--(2.6000,0.8000)--(2.6000,1.0000)--(2.4000,1.0000)--(2.4000,1.2000)--(2.2000,1.2000)--(2.2000,1.4000)--(2.0000,1.4000)--(2.0000,1.6000)--(1.8000,1.6000)--(1.8000,1.8000)--(1.6000,1.8000)--(1.6000,2.0000)--(1.4000,2.0000)--(1.4000,2.2000)--(1.2000,2.2000)--(1.2000,2.4000)--(1.0000,2.4000)--(1.0000,2.6000)--(0.8000,2.6000)--(0.8000,2.8000)--(0.8000,3.0000)--(0.8000,3.2000)--(1.0000,3.2000)--(1.2000,3.2000)--(1.4000,3.2000)--(1.4000,3.0000)--(1.6000,3.0000)--(1.6000,2.8000)--(1.8000,2.8000)--(1.8000,2.6000)--(2.0000,2.6000)--(2.0000,2.4000)--(2.2000,2.4000)--(2.2000,2.2000)--(2.4000,2.2000)--(2.4000,2.0000)--(2.6000,2.0000)--(2.6000,1.8000)--(2.8000,1.8000)--(2.8000,1.6000)--(3.0000,1.6000)--(3.0000,1.4000)--(3.2000,1.4000)--(3.2000,1.2000)--(3.4000,1.2000)--(3.4000,1.0000)--(3.6000,1.0000)--(3.6000,0.8000)--(3.8000,0.8000)--(3.8000,0.6000)--(3.6000,0.6000)--(3.6000,0.4000)--(3.4000,0.4000)--(3.4000,0.2000)--(3.2000,0.2000);
\draw [gray!10](0.8000,2.6000)--(0.8000,3.2000);
\draw [gray!10](1.0000,2.4000)--(1.0000,3.2000);
\draw [gray!10](1.2000,2.2000)--(1.2000,3.2000);
\draw [gray!10](1.4000,2.0000)--(1.4000,3.2000);
\draw [gray!10](1.6000,1.8000)--(1.6000,3.0000);
\draw [gray!10](1.8000,1.6000)--(1.8000,2.8000);
\draw [gray!10](2.0000,1.4000)--(2.0000,2.6000);
\draw [gray!10](2.2000,1.2000)--(2.2000,2.4000);
\draw [gray!10](2.4000,1.0000)--(2.4000,2.2000);
\draw [gray!10](2.6000,0.8000)--(2.6000,2.0000);
\draw [gray!10](2.8000,0.6000)--(2.8000,1.8000);
\draw [gray!10](3.0000,0.4000)--(3.0000,1.6000);
\draw [gray!10](3.2000,0.2000)--(3.2000,1.4000);
\draw [gray!10](3.4000,0.2000)--(3.4000,1.2000);
\draw [gray!10](3.6000,0.4000)--(3.6000,1.0000);
\draw [gray!10](3.2000,0.2000)--(3.4000,0.2000);
\draw [gray!10](3.0000,0.4000)--(3.6000,0.4000);
\draw [gray!10](2.8000,0.6000)--(3.8000,0.6000);
\draw [gray!10](2.6000,0.8000)--(3.8000,0.8000);
\draw [gray!10](2.4000,1.0000)--(3.6000,1.0000);
\draw [gray!10](2.2000,1.2000)--(3.4000,1.2000);
\draw [gray!10](2.0000,1.4000)--(3.2000,1.4000);
\draw [gray!10](1.8000,1.6000)--(3.0000,1.6000);
\draw [gray!10](1.6000,1.8000)--(2.8000,1.8000);
\draw [gray!10](1.4000,2.0000)--(2.6000,2.0000);
\draw [gray!10](1.2000,2.2000)--(2.4000,2.2000);
\draw [gray!10](1.0000,2.4000)--(2.2000,2.4000);
\draw [gray!10](0.8000,2.6000)--(2.0000,2.6000);
\draw [gray!10](0.8000,2.8000)--(1.8000,2.8000);
\draw [gray!10](0.8000,3.0000)--(1.6000,3.0000);
\draw [black](3.2000,0.2000)--(3.2000,0.4000)--(3.0000,0.4000)--(3.0000,0.6000)--(2.8000,0.6000)--(2.8000,0.8000)--(2.6000,0.8000)--(2.6000,1.0000)--(2.4000,1.0000)--(2.4000,1.2000)--(2.2000,1.2000)--(2.2000,1.4000)--(2.0000,1.4000)--(2.0000,1.6000)--(1.8000,1.6000)--(1.8000,1.8000)--(1.6000,1.8000)--(1.6000,2.0000)--(1.4000,2.0000)--(1.4000,2.2000)--(1.2000,2.2000)--(1.2000,2.4000)--(1.0000,2.4000)--(1.0000,2.6000)--(0.8000,2.6000)--(0.8000,2.8000)--(0.8000,3.0000)--(0.8000,3.2000)--(1.0000,3.2000)--(1.2000,3.2000)--(1.4000,3.2000)--(1.4000,3.0000)--(1.6000,3.0000)--(1.6000,2.8000)--(1.8000,2.8000)--(1.8000,2.6000)--(2.0000,2.6000)--(2.0000,2.4000)--(2.2000,2.4000)--(2.2000,2.2000)--(2.4000,2.2000)--(2.4000,2.0000)--(2.6000,2.0000)--(2.6000,1.8000)--(2.8000,1.8000)--(2.8000,1.6000)--(3.0000,1.6000)--(3.0000,1.4000)--(3.2000,1.4000)--(3.2000,1.2000)--(3.4000,1.2000)--(3.4000,1.0000)--(3.6000,1.0000)--(3.6000,0.8000)--(3.8000,0.8000);
\draw (0.8000,3.2000)--(0.8000,-0.2000);
\draw (0.8000,3.2000)--(4.2000,3.2000);
\draw [dashed] (0.8000,3.2000)--(4.0000,0.0000);
\draw (3.2,  2.0) node [] {$\lambda$};
\draw (2.0,  0.8) node [] {$\lambda$};
	\draw (2.5,-1) node [] {SCPa $\sim \frac{0.146}{n^{0.85}} e^{2.14 \sqrt{n}}$};
\end{scope}
\begin{scope}[xshift=180, yshift=0]
\fill [gray!30](3.2000,0.2000)--(3.2000,0.4000)--(3.0000,0.4000)--(3.0000,0.6000)--(2.8000,0.6000)--(2.8000,0.8000)--(2.6000,0.8000)--(2.6000,1.0000)--(2.4000,1.0000)--(2.4000,1.2000)--(2.2000,1.2000)--(2.2000,1.4000)--(2.0000,1.4000)--(2.0000,1.6000)--(1.8000,1.6000)--(1.8000,1.8000)--(1.6000,1.8000)--(1.6000,2.0000)--(1.4000,2.0000)--(1.4000,2.2000)--(1.2000,2.2000)--(1.2000,2.4000)--(1.0000,2.4000)--(1.0000,2.6000)--(1.0000,2.8000)--(1.2000,2.8000)--(1.2000,3.0000)--(1.4000,3.0000)--(1.6000,3.0000)--(1.6000,2.8000)--(1.8000,2.8000)--(1.8000,2.6000)--(2.0000,2.6000)--(2.0000,2.4000)--(2.2000,2.4000)--(2.2000,2.2000)--(2.4000,2.2000)--(2.4000,2.0000)--(2.6000,2.0000)--(2.6000,1.8000)--(2.8000,1.8000)--(2.8000,1.6000)--(3.0000,1.6000)--(3.0000,1.4000)--(3.2000,1.4000)--(3.2000,1.2000)--(3.4000,1.2000)--(3.4000,1.0000)--(3.6000,1.0000)--(3.6000,0.8000)--(3.8000,0.8000)--(3.8000,0.6000)--(3.6000,0.6000)--(3.6000,0.4000)--(3.4000,0.4000)--(3.4000,0.2000)--(3.2000,0.2000);
\draw [gray!10](1.0000,2.4000)--(1.0000,2.8000);
\draw [gray!10](1.2000,2.2000)--(1.2000,3.0000);
\draw [gray!10](1.4000,2.0000)--(1.4000,3.0000);
\draw [gray!10](1.6000,1.8000)--(1.6000,3.0000);
\draw [gray!10](1.8000,1.6000)--(1.8000,2.8000);
\draw [gray!10](2.0000,1.4000)--(2.0000,2.6000);
\draw [gray!10](2.2000,1.2000)--(2.2000,2.4000);
\draw [gray!10](2.4000,1.0000)--(2.4000,2.2000);
\draw [gray!10](2.6000,0.8000)--(2.6000,2.0000);
\draw [gray!10](2.8000,0.6000)--(2.8000,1.8000);
\draw [gray!10](3.0000,0.4000)--(3.0000,1.6000);
\draw [gray!10](3.2000,0.2000)--(3.2000,1.4000);
\draw [gray!10](3.4000,0.2000)--(3.4000,1.2000);
\draw [gray!10](3.6000,0.4000)--(3.6000,1.0000);
\draw [gray!10](3.2000,0.2000)--(3.4000,0.2000);
\draw [gray!10](3.0000,0.4000)--(3.6000,0.4000);
\draw [gray!10](2.8000,0.6000)--(3.8000,0.6000);
\draw [gray!10](2.6000,0.8000)--(3.8000,0.8000);
\draw [gray!10](2.4000,1.0000)--(3.6000,1.0000);
\draw [gray!10](2.2000,1.2000)--(3.4000,1.2000);
\draw [gray!10](2.0000,1.4000)--(3.2000,1.4000);
\draw [gray!10](1.8000,1.6000)--(3.0000,1.6000);
\draw [gray!10](1.6000,1.8000)--(2.8000,1.8000);
\draw [gray!10](1.4000,2.0000)--(2.6000,2.0000);
\draw [gray!10](1.2000,2.2000)--(2.4000,2.2000);
\draw [gray!10](1.0000,2.4000)--(2.2000,2.4000);
\draw [gray!10](1.0000,2.6000)--(2.0000,2.6000);
\draw [gray!10](1.0000,2.8000)--(1.8000,2.8000);
\draw [black](3.2000,0.2000)--(3.2000,0.4000)--(3.0000,0.4000)--(3.0000,0.6000)--(2.8000,0.6000)--(2.8000,0.8000)--(2.6000,0.8000)--(2.6000,1.0000)--(2.4000,1.0000)--(2.4000,1.2000)--(2.2000,1.2000)--(2.2000,1.4000)--(2.0000,1.4000)--(2.0000,1.6000)--(1.8000,1.6000)--(1.8000,1.8000)--(1.6000,1.8000)--(1.6000,2.0000)--(1.4000,2.0000)--(1.4000,2.2000)--(1.2000,2.2000)--(1.2000,2.4000)--(1.0000,2.4000)--(1.0000,2.6000)--(1.0000,2.8000)--(1.2000,2.8000)--(1.2000,3.0000)--(1.4000,3.0000)--(1.6000,3.0000)--(1.6000,2.8000)--(1.8000,2.8000)--(1.8000,2.6000)--(2.0000,2.6000)--(2.0000,2.4000)--(2.2000,2.4000)--(2.2000,2.2000)--(2.4000,2.2000)--(2.4000,2.0000)--(2.6000,2.0000)--(2.6000,1.8000)--(2.8000,1.8000)--(2.8000,1.6000)--(3.0000,1.6000)--(3.0000,1.4000)--(3.2000,1.4000)--(3.2000,1.2000)--(3.4000,1.2000)--(3.4000,1.0000)--(3.6000,1.0000)--(3.6000,0.8000)--(3.8000,0.8000);
\draw (1.0000,3.0000)--(1.0000,-0.2000);
\draw (1.0000,3.0000)--(4.2000,3.0000);
\draw [dashed] (1.0000,3.0000)--(4.0000,0.0000);
\draw (3.2,  2.0) node [] {$\lambda$};
\draw (2.0,  0.8) node [] {$\lambda$};
	\draw (2.5,-1) node [] {SCPb $\sim \frac{0.131}{n^{1.05}} e^{2.14 \sqrt{n}}$};
\end{scope}
\begin{scope}[xshift=360, yshift=0]
\fill [gray!30](3.2000,0.2000)--(3.2000,0.4000)--(3.0000,0.4000)--(3.0000,0.6000)--(2.8000,0.6000)--(2.8000,0.8000)--(2.6000,0.8000)--(2.6000,1.0000)--(2.4000,1.0000)--(2.4000,1.2000)--(2.2000,1.2000)--(2.2000,1.4000)--(2.0000,1.4000)--(2.0000,1.6000)--(1.8000,1.6000)--(1.8000,1.8000)--(1.6000,1.8000)--(1.6000,2.0000)--(1.4000,2.0000)--(1.4000,2.2000)--(1.2000,2.2000)--(1.2000,2.4000)--(1.0000,2.4000)--(1.0000,2.6000)--(1.2000,2.6000)--(1.4000,2.6000)--(1.4000,2.8000)--(1.4000,3.0000)--(1.6000,3.0000)--(1.6000,2.8000)--(1.8000,2.8000)--(1.8000,2.6000)--(2.0000,2.6000)--(2.0000,2.4000)--(2.2000,2.4000)--(2.2000,2.2000)--(2.4000,2.2000)--(2.4000,2.0000)--(2.6000,2.0000)--(2.6000,1.8000)--(2.8000,1.8000)--(2.8000,1.6000)--(3.0000,1.6000)--(3.0000,1.4000)--(3.2000,1.4000)--(3.2000,1.2000)--(3.4000,1.2000)--(3.4000,1.0000)--(3.6000,1.0000)--(3.6000,0.8000)--(3.8000,0.8000)--(3.8000,0.6000)--(3.6000,0.6000)--(3.6000,0.4000)--(3.4000,0.4000)--(3.4000,0.2000)--(3.2000,0.2000);
\draw [gray!10](1.0000,2.4000)--(1.0000,2.6000);
\draw [gray!10](1.2000,2.2000)--(1.2000,2.6000);
\draw [gray!10](1.4000,2.0000)--(1.4000,3.0000);
\draw [gray!10](1.6000,1.8000)--(1.6000,3.0000);
\draw [gray!10](1.8000,1.6000)--(1.8000,2.8000);
\draw [gray!10](2.0000,1.4000)--(2.0000,2.6000);
\draw [gray!10](2.2000,1.2000)--(2.2000,2.4000);
\draw [gray!10](2.4000,1.0000)--(2.4000,2.2000);
\draw [gray!10](2.6000,0.8000)--(2.6000,2.0000);
\draw [gray!10](2.8000,0.6000)--(2.8000,1.8000);
\draw [gray!10](3.0000,0.4000)--(3.0000,1.6000);
\draw [gray!10](3.2000,0.2000)--(3.2000,1.4000);
\draw [gray!10](3.4000,0.2000)--(3.4000,1.2000);
\draw [gray!10](3.6000,0.4000)--(3.6000,1.0000);
\draw [gray!10](3.2000,0.2000)--(3.4000,0.2000);
\draw [gray!10](3.0000,0.4000)--(3.6000,0.4000);
\draw [gray!10](2.8000,0.6000)--(3.8000,0.6000);
\draw [gray!10](2.6000,0.8000)--(3.8000,0.8000);
\draw [gray!10](2.4000,1.0000)--(3.6000,1.0000);
\draw [gray!10](2.2000,1.2000)--(3.4000,1.2000);
\draw [gray!10](2.0000,1.4000)--(3.2000,1.4000);
\draw [gray!10](1.8000,1.6000)--(3.0000,1.6000);
\draw [gray!10](1.6000,1.8000)--(2.8000,1.8000);
\draw [gray!10](1.4000,2.0000)--(2.6000,2.0000);
\draw [gray!10](1.2000,2.2000)--(2.4000,2.2000);
\draw [gray!10](1.0000,2.4000)--(2.2000,2.4000);
\draw [gray!10](1.0000,2.6000)--(2.0000,2.6000);
\draw [gray!10](1.4000,2.8000)--(1.8000,2.8000);
\draw [black](3.2000,0.2000)--(3.2000,0.4000)--(3.0000,0.4000)--(3.0000,0.6000)--(2.8000,0.6000)--(2.8000,0.8000)--(2.6000,0.8000)--(2.6000,1.0000)--(2.4000,1.0000)--(2.4000,1.2000)--(2.2000,1.2000)--(2.2000,1.4000)--(2.0000,1.4000)--(2.0000,1.6000)--(1.8000,1.6000)--(1.8000,1.8000)--(1.6000,1.8000)--(1.6000,2.0000)--(1.4000,2.0000)--(1.4000,2.2000)--(1.2000,2.2000)--(1.2000,2.4000)--(1.0000,2.4000)--(1.0000,2.6000)--(1.2000,2.6000)--(1.4000,2.6000)--(1.4000,2.8000)--(1.4000,3.0000)--(1.6000,3.0000)--(1.6000,2.8000)--(1.8000,2.8000)--(1.8000,2.6000)--(2.0000,2.6000)--(2.0000,2.4000)--(2.2000,2.4000)--(2.2000,2.2000)--(2.4000,2.2000)--(2.4000,2.0000)--(2.6000,2.0000)--(2.6000,1.8000)--(2.8000,1.8000)--(2.8000,1.6000)--(3.0000,1.6000)--(3.0000,1.4000)--(3.2000,1.4000)--(3.2000,1.2000)--(3.4000,1.2000)--(3.4000,1.0000)--(3.6000,1.0000)--(3.6000,0.8000)--(3.8000,0.8000);
\draw (1.0000,3.0000)--(1.0000,-0.2000);
\draw (1.0000,3.0000)--(4.2000,3.0000);
\draw [dashed] (1.0000,3.0000)--(4.0000,0.0000);
\draw (3.2,  2.0) node [] {$\lambda$};
\draw (2.0,  0.8) node [] {$\lambda$};
	\draw (2.5,-1) node [] {SCPc $\sim \frac{0.116}{n^{1.15}} e^{2.14 \sqrt{n}}$};
\end{scope}
\end{tikzpicture}
\\\begin{tikzpicture}[scale=0.6]
\begin{scope}[xshift=0, yshift=0]
\fill [gray!30](3.2000,0.2000)--(3.2000,0.4000)--(3.0000,0.4000)--(3.0000,0.6000)--(2.8000,0.6000)--(2.8000,0.8000)--(2.6000,0.8000)--(2.6000,1.0000)--(2.4000,1.0000)--(2.4000,1.2000)--(2.2000,1.2000)--(2.2000,1.4000)--(2.0000,1.4000)--(2.0000,1.6000)--(1.8000,1.6000)--(1.8000,1.8000)--(1.6000,1.8000)--(1.6000,2.0000)--(1.4000,2.0000)--(1.4000,2.2000)--(1.2000,2.2000)--(1.2000,2.4000)--(1.0000,2.4000)--(1.0000,2.6000)--(1.2000,2.6000)--(1.4000,2.6000)--(1.6000,2.6000)--(1.6000,2.4000)--(1.8000,2.4000)--(1.8000,2.2000)--(2.0000,2.2000)--(2.0000,2.0000)--(2.2000,2.0000)--(2.2000,1.8000)--(2.4000,1.8000)--(2.4000,1.6000)--(2.6000,1.6000)--(2.6000,1.4000)--(2.8000,1.4000)--(2.8000,1.2000)--(3.0000,1.2000)--(3.0000,1.0000)--(3.2000,1.0000)--(3.2000,0.8000)--(3.4000,0.8000)--(3.4000,0.6000)--(3.6000,0.6000)--(3.6000,0.4000)--(3.4000,0.4000)--(3.4000,0.2000)--(3.2000,0.2000);
\draw [gray!10](1.0000,2.4000)--(1.0000,2.6000);
\draw [gray!10](1.2000,2.2000)--(1.2000,2.6000);
\draw [gray!10](1.4000,2.0000)--(1.4000,2.6000);
\draw [gray!10](1.6000,1.8000)--(1.6000,2.6000);
\draw [gray!10](1.8000,1.6000)--(1.8000,2.4000);
\draw [gray!10](2.0000,1.4000)--(2.0000,2.2000);
\draw [gray!10](2.2000,1.2000)--(2.2000,2.0000);
\draw [gray!10](2.4000,1.0000)--(2.4000,1.8000);
\draw [gray!10](2.6000,0.8000)--(2.6000,1.6000);
\draw [gray!10](2.8000,0.6000)--(2.8000,1.4000);
\draw [gray!10](3.0000,0.4000)--(3.0000,1.2000);
\draw [gray!10](3.2000,0.2000)--(3.2000,1.0000);
\draw [gray!10](3.4000,0.2000)--(3.4000,0.8000);
\draw [gray!10](3.2000,0.2000)--(3.4000,0.2000);
\draw [gray!10](3.0000,0.4000)--(3.6000,0.4000);
\draw [gray!10](2.8000,0.6000)--(3.6000,0.6000);
\draw [gray!10](2.6000,0.8000)--(3.4000,0.8000);
\draw [gray!10](2.4000,1.0000)--(3.2000,1.0000);
\draw [gray!10](2.2000,1.2000)--(3.0000,1.2000);
\draw [gray!10](2.0000,1.4000)--(2.8000,1.4000);
\draw [gray!10](1.8000,1.6000)--(2.6000,1.6000);
\draw [gray!10](1.6000,1.8000)--(2.4000,1.8000);
\draw [gray!10](1.4000,2.0000)--(2.2000,2.0000);
\draw [gray!10](1.2000,2.2000)--(2.0000,2.2000);
\draw [gray!10](1.0000,2.4000)--(1.8000,2.4000);
\draw [black](3.2000,0.2000)--(3.2000,0.4000)--(3.0000,0.4000)--(3.0000,0.6000)--(2.8000,0.6000)--(2.8000,0.8000)--(2.6000,0.8000)--(2.6000,1.0000)--(2.4000,1.0000)--(2.4000,1.2000)--(2.2000,1.2000)--(2.2000,1.4000)--(2.0000,1.4000)--(2.0000,1.6000)--(1.8000,1.6000)--(1.8000,1.8000)--(1.6000,1.8000)--(1.6000,2.0000)--(1.4000,2.0000)--(1.4000,2.2000)--(1.2000,2.2000)--(1.2000,2.4000)--(1.0000,2.4000)--(1.0000,2.6000)--(1.2000,2.6000)--(1.4000,2.6000)--(1.6000,2.6000)--(1.6000,2.4000)--(1.8000,2.4000)--(1.8000,2.2000)--(2.0000,2.2000)--(2.0000,2.0000)--(2.2000,2.0000)--(2.2000,1.8000)--(2.4000,1.8000)--(2.4000,1.6000)--(2.6000,1.6000)--(2.6000,1.4000)--(2.8000,1.4000)--(2.8000,1.2000)--(3.0000,1.2000)--(3.0000,1.0000)--(3.2000,1.0000)--(3.2000,0.8000)--(3.4000,0.8000)--(3.4000,0.6000)--(3.6000,0.6000);
\draw (1.0000,2.6000)--(1.0000,-0.2000);
\draw (1.0000,2.6000)--(4.0000,2.6000);
\draw (2.5,-1) node [] {DSPPa $\sim \frac{0.110}n e^{2.77 \sqrt{n}}$};
\end{scope}
\begin{scope}[xshift=180, yshift=0]
\fill [gray!30](3.2000,0.2000)--(3.2000,0.4000)--(3.0000,0.4000)--(3.0000,0.6000)--(2.8000,0.6000)--(2.8000,0.8000)--(2.6000,0.8000)--(2.6000,1.0000)--(2.4000,1.0000)--(2.4000,1.2000)--(2.2000,1.2000)--(2.2000,1.4000)--(2.0000,1.4000)--(2.0000,1.6000)--(1.8000,1.6000)--(1.8000,1.8000)--(1.6000,1.8000)--(1.6000,2.0000)--(1.4000,2.0000)--(1.4000,2.2000)--(1.2000,2.2000)--(1.2000,2.4000)--(1.0000,2.4000)--(1.0000,2.6000)--(1.0000,2.8000)--(1.2000,2.8000)--(1.4000,2.8000)--(1.4000,2.6000)--(1.6000,2.6000)--(1.6000,2.4000)--(1.8000,2.4000)--(1.8000,2.2000)--(2.0000,2.2000)--(2.0000,2.0000)--(2.2000,2.0000)--(2.2000,1.8000)--(2.4000,1.8000)--(2.4000,1.6000)--(2.6000,1.6000)--(2.6000,1.4000)--(2.8000,1.4000)--(2.8000,1.2000)--(3.0000,1.2000)--(3.0000,1.0000)--(3.2000,1.0000)--(3.2000,0.8000)--(3.4000,0.8000)--(3.4000,0.6000)--(3.6000,0.6000)--(3.6000,0.4000)--(3.4000,0.4000)--(3.4000,0.2000)--(3.2000,0.2000);
\draw [gray!10](1.0000,2.4000)--(1.0000,2.8000);
\draw [gray!10](1.2000,2.2000)--(1.2000,2.8000);
\draw [gray!10](1.4000,2.0000)--(1.4000,2.8000);
\draw [gray!10](1.6000,1.8000)--(1.6000,2.6000);
\draw [gray!10](1.8000,1.6000)--(1.8000,2.4000);
\draw [gray!10](2.0000,1.4000)--(2.0000,2.2000);
\draw [gray!10](2.2000,1.2000)--(2.2000,2.0000);
\draw [gray!10](2.4000,1.0000)--(2.4000,1.8000);
\draw [gray!10](2.6000,0.8000)--(2.6000,1.6000);
\draw [gray!10](2.8000,0.6000)--(2.8000,1.4000);
\draw [gray!10](3.0000,0.4000)--(3.0000,1.2000);
\draw [gray!10](3.2000,0.2000)--(3.2000,1.0000);
\draw [gray!10](3.4000,0.2000)--(3.4000,0.8000);
\draw [gray!10](3.2000,0.2000)--(3.4000,0.2000);
\draw [gray!10](3.0000,0.4000)--(3.6000,0.4000);
\draw [gray!10](2.8000,0.6000)--(3.6000,0.6000);
\draw [gray!10](2.6000,0.8000)--(3.4000,0.8000);
\draw [gray!10](2.4000,1.0000)--(3.2000,1.0000);
\draw [gray!10](2.2000,1.2000)--(3.0000,1.2000);
\draw [gray!10](2.0000,1.4000)--(2.8000,1.4000);
\draw [gray!10](1.8000,1.6000)--(2.6000,1.6000);
\draw [gray!10](1.6000,1.8000)--(2.4000,1.8000);
\draw [gray!10](1.4000,2.0000)--(2.2000,2.0000);
\draw [gray!10](1.2000,2.2000)--(2.0000,2.2000);
\draw [gray!10](1.0000,2.4000)--(1.8000,2.4000);
\draw [gray!10](1.0000,2.6000)--(1.6000,2.6000);
\draw [black](3.2000,0.2000)--(3.2000,0.4000)--(3.0000,0.4000)--(3.0000,0.6000)--(2.8000,0.6000)--(2.8000,0.8000)--(2.6000,0.8000)--(2.6000,1.0000)--(2.4000,1.0000)--(2.4000,1.2000)--(2.2000,1.2000)--(2.2000,1.4000)--(2.0000,1.4000)--(2.0000,1.6000)--(1.8000,1.6000)--(1.8000,1.8000)--(1.6000,1.8000)--(1.6000,2.0000)--(1.4000,2.0000)--(1.4000,2.2000)--(1.2000,2.2000)--(1.2000,2.4000)--(1.0000,2.4000)--(1.0000,2.6000)--(1.0000,2.8000)--(1.2000,2.8000)--(1.4000,2.8000)--(1.4000,2.6000)--(1.6000,2.6000)--(1.6000,2.4000)--(1.8000,2.4000)--(1.8000,2.2000)--(2.0000,2.2000)--(2.0000,2.0000)--(2.2000,2.0000)--(2.2000,1.8000)--(2.4000,1.8000)--(2.4000,1.6000)--(2.6000,1.6000)--(2.6000,1.4000)--(2.8000,1.4000)--(2.8000,1.2000)--(3.0000,1.2000)--(3.0000,1.0000)--(3.2000,1.0000)--(3.2000,0.8000)--(3.4000,0.8000)--(3.4000,0.6000)--(3.6000,0.6000);
\draw (1.0000,2.8000)--(1.0000,-0.2000);
\draw (1.0000,2.8000)--(4.0000,2.8000);
\draw (2.5,-1) node [] {DSPPb $\sim \frac{0.138}n e^{2.77 \sqrt{n}}$};
\end{scope}
\begin{scope}[xshift=360, yshift=0]
\fill [gray!30](3.2000,0.2000)--(3.2000,0.4000)--(3.0000,0.4000)--(3.0000,0.6000)--(2.8000,0.6000)--(2.8000,0.8000)--(2.6000,0.8000)--(2.6000,1.0000)--(2.4000,1.0000)--(2.4000,1.2000)--(2.2000,1.2000)--(2.2000,1.4000)--(2.0000,1.4000)--(2.0000,1.6000)--(1.8000,1.6000)--(1.8000,1.8000)--(1.6000,1.8000)--(1.6000,2.0000)--(1.4000,2.0000)--(1.4000,2.2000)--(1.2000,2.2000)--(1.2000,2.4000)--(1.0000,2.4000)--(1.0000,2.6000)--(1.2000,2.6000)--(1.2000,2.8000)--(1.4000,2.8000)--(1.4000,2.6000)--(1.6000,2.6000)--(1.6000,2.4000)--(1.8000,2.4000)--(1.8000,2.2000)--(2.0000,2.2000)--(2.0000,2.0000)--(2.2000,2.0000)--(2.2000,1.8000)--(2.4000,1.8000)--(2.4000,1.6000)--(2.6000,1.6000)--(2.6000,1.4000)--(2.8000,1.4000)--(2.8000,1.2000)--(3.0000,1.2000)--(3.0000,1.0000)--(3.2000,1.0000)--(3.2000,0.8000)--(3.4000,0.8000)--(3.4000,0.6000)--(3.6000,0.6000)--(3.6000,0.4000)--(3.4000,0.4000)--(3.4000,0.2000)--(3.2000,0.2000);
\draw [gray!10](1.0000,2.4000)--(1.0000,2.6000);
\draw [gray!10](1.2000,2.2000)--(1.2000,2.8000);
\draw [gray!10](1.4000,2.0000)--(1.4000,2.8000);
\draw [gray!10](1.6000,1.8000)--(1.6000,2.6000);
\draw [gray!10](1.8000,1.6000)--(1.8000,2.4000);
\draw [gray!10](2.0000,1.4000)--(2.0000,2.2000);
\draw [gray!10](2.2000,1.2000)--(2.2000,2.0000);
\draw [gray!10](2.4000,1.0000)--(2.4000,1.8000);
\draw [gray!10](2.6000,0.8000)--(2.6000,1.6000);
\draw [gray!10](2.8000,0.6000)--(2.8000,1.4000);
\draw [gray!10](3.0000,0.4000)--(3.0000,1.2000);
\draw [gray!10](3.2000,0.2000)--(3.2000,1.0000);
\draw [gray!10](3.4000,0.2000)--(3.4000,0.8000);
\draw [gray!10](3.2000,0.2000)--(3.4000,0.2000);
\draw [gray!10](3.0000,0.4000)--(3.6000,0.4000);
\draw [gray!10](2.8000,0.6000)--(3.6000,0.6000);
\draw [gray!10](2.6000,0.8000)--(3.4000,0.8000);
\draw [gray!10](2.4000,1.0000)--(3.2000,1.0000);
\draw [gray!10](2.2000,1.2000)--(3.0000,1.2000);
\draw [gray!10](2.0000,1.4000)--(2.8000,1.4000);
\draw [gray!10](1.8000,1.6000)--(2.6000,1.6000);
\draw [gray!10](1.6000,1.8000)--(2.4000,1.8000);
\draw [gray!10](1.4000,2.0000)--(2.2000,2.0000);
\draw [gray!10](1.2000,2.2000)--(2.0000,2.2000);
\draw [gray!10](1.0000,2.4000)--(1.8000,2.4000);
\draw [gray!10](1.0000,2.6000)--(1.6000,2.6000);
\draw [black](3.2000,0.2000)--(3.2000,0.4000)--(3.0000,0.4000)--(3.0000,0.6000)--(2.8000,0.6000)--(2.8000,0.8000)--(2.6000,0.8000)--(2.6000,1.0000)--(2.4000,1.0000)--(2.4000,1.2000)--(2.2000,1.2000)--(2.2000,1.4000)--(2.0000,1.4000)--(2.0000,1.6000)--(1.8000,1.6000)--(1.8000,1.8000)--(1.6000,1.8000)--(1.6000,2.0000)--(1.4000,2.0000)--(1.4000,2.2000)--(1.2000,2.2000)--(1.2000,2.4000)--(1.0000,2.4000)--(1.0000,2.6000)--(1.2000,2.6000)--(1.2000,2.8000)--(1.4000,2.8000)--(1.4000,2.6000)--(1.6000,2.6000)--(1.6000,2.4000)--(1.8000,2.4000)--(1.8000,2.2000)--(2.0000,2.2000)--(2.0000,2.0000)--(2.2000,2.0000)--(2.2000,1.8000)--(2.4000,1.8000)--(2.4000,1.6000)--(2.6000,1.6000)--(2.6000,1.4000)--(2.8000,1.4000)--(2.8000,1.2000)--(3.0000,1.2000)--(3.0000,1.0000)--(3.2000,1.0000)--(3.2000,0.8000)--(3.4000,0.8000)--(3.4000,0.6000)--(3.6000,0.6000);
\draw (1.0000,2.8000)--(1.0000,-0.2000);
\draw (1.0000,2.8000)--(4.0000,2.8000);
\draw (2.5,-1) node [] {DSPPc $\sim \frac{0.174}n e^{2.77 \sqrt{n}}$};
\end{scope}
\end{tikzpicture}
\end{align*}

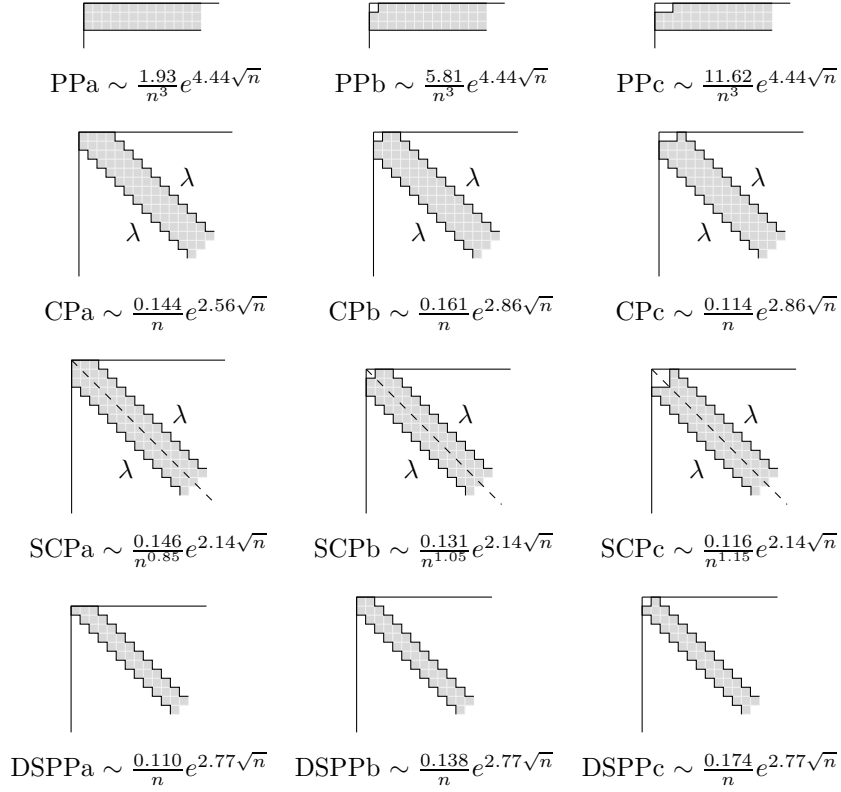
\centerline{Fig. 2. Asymptotic formulas for various kinds of defective plane partitions.}
\medskip

The skew doubled shifted plane partitions have some nice properties.
(1) The order of the asymptotic formula depends only on the width of the doubled shifted plane partition, not on the profile (the skew zone) itself. 
The similar property holds for ordinary plane partitions. 
We may think that this is natural by intuition. 
However, the cylindric partitions (CP and SCP) show that this is not always the case. 
(2) We empirically observe that, the asymptotic formula for doubled shifted plane partitions gives already good 
approximative values for the numbers of DSPP, even for small integer $n$.
While the asymptotic formula for PP needs a large integer $n$
to produce an acceptable value, as shown in the following table.
$$
\begin{tabular}{|c | r  r  r  r|}
	\hline
	$n$      & 5   & 10 & 15 & 20      \\
	\hline
	\#PPa  & 21   & 319 & 3032 &  22371  \\
	Asymptotic & $\sim$ 319& $\sim$ 2449 & $\sim$ 17062 & $\sim$ 103112    \\
	\hline
	\#CPa & 7   & 42 & 176 & 627   \\
	Asymptotic  & $\sim$ 8& $\sim$ 48 & $\sim$ 198 & $\sim$ 692    \\
	\hline
	\#SCPa & 4   & 17 & 56 & 161   \\
	Asymptotic    & $\sim$ 4& $\sim$ 18 & $\sim$ 59 & $\sim$ 169    \\
	\hline
	\#DSPPa & 9   & 64 & 314 & 1244   \\
	Asymptotic   & $\sim$ 10& $\sim$ 70 & $\sim$ 336 & $\sim$ 1325    \\
	\hline
\end{tabular}
$$

It is amazing how the orders of the asymptotic formulas for CP and SCP differ. 
For the CP,
the exponents of $n$ in the denominator are always $1$, but the
exponents of $e$ differ. While for the SCP, the exponents of $n$
differ, but the exponents of $n$ are constant. Let us summarize these observations in the following table.

$$
\begin{tabular}{ |c| c |c| c |}
	\hline
	& $n^{\text{Const}}$ & $e^{\text{Const}\sqrt{n}}$ & Fast Convergence \\
	\hline
	PP & Yes & Yes & No \\
	CP & Yes & No & Yes \\
	SCP & No & Yes & Yes \\
	DSPP & Yes & Yes & Yes\\
	\hline
\end{tabular}
$$

\medskip
The rest of the paper is arranged in the following way. 
In Section~\ref{sec:main-results} we establish the complete summation formula
for Schur processes.
The basic notation and the trace generating functions
for skew doubled shifted plane partitions can be found in
Section~\ref{sec:SkRPP}. 
After recalling some useful theorems for asymptotic formulas in Section~\ref{sec:Asym}, we compute the generating functions and the asymptotic formulas
for the numbers of skew doubled shifted plane partitions and symmetric cylindric partitions, in Sections \ref{sec:gfRPP} and \ref{sec:gfSyCPP}, respectively.
% >>>

\section{Summation formulas for skew Schur functions}\label{sec:main-results} % <<<

For the definitions and basic properties of skew Schur functions we refer to 
the books \cite{Macdonald1995,StanleyEC2,StanleyEC1}.
Let 
\begin{align*}
	\Psi(X,Y) &= \prod_{i,j} (1-x_i y_j)^{-1},\\ 
    \Phi(X) &= \prod_i (1-x_i)^{-1} \prod_{i<j} (1-x_i x_j)^{-1},
\end{align*}
where $X=\{x_1, x_2, \ldots\}$ and $Y=\{y_1, y_2, \ldots \}$ are two alphabets.
Each $\pm 1$-sequence $\delta=(\delta_i)_{1\leq i\leq h}$ of
length $h\geq 1$ is called a {\it profile}. Let $|\delta|_1$ 
(resp. $|\delta|_{-1}$) be the number of letters $1$ (resp.  $-1$) in $\delta$.
Therefore, $h=|\delta|_1+|\delta|_{-1}$.
The following theorem contains three fundamental summation formulas
for skew Schur functions, namely,
the {\it open summation formula} \eqref{eq:MainOpen},
the {\it cylindric summation formula} \eqref{eq:MainCylinder}
and
the {\it complete summation formula} \eqref{eq:MainComplete}.
The open and cylindric formulas 
have already been derived by Okounkov, Reshetikhin, Borodin, Corteel, Savelief, Vuleti\'c and Langer \cite{Borodin2007,Borodin2011,CorteelSaveliefVuletic,  Langer2013A, Langer2013B,Okounkov2001,Okounkov2003, Vuletic2007,Vuletic2009}.
For convenience, they are also reproduced next.
Our main contribution is the complete summation formula \eqref{eq:MainComplete}.

\begin{thm}\label{th:MainZ}
Let $h$ be a positive integer,
	$\delta=(\delta_i)_{1\leq i\leq h}$ be a profile of length $h$, 
	and $Z^1,\ldots,Z^h$ be a sequence of alphabets.
	Write $\mathbf{Z}:=\mathbf{Z}^\delta_-+\mathbf{Z}^\delta_+=\sum_{1\leq i\leq h} Z^i$ as the union of $Z^1, Z^2, \ldots, Z^{h}$, $\mathbf{Z}^\delta_- := \sum_{i:\, \delta_i=-1} Z^i$,
	and $\mathbf{Z}^\delta_+ := \sum_{i:\, \delta_i=1} Z^i$.
	For a sequence of partitions $\la^0,\la^1,\la^2,\ldots,\la^h$, let $s^{\delta}_i$ denote the skew Schur function $s_{\la^{i}/\la^{i-1}}$ if $\delta_i=1$ and $s_{\la^{i-1}/\la^{i}}$ if $\delta_i=-1$. 
We have
\begin{align} 
	\sum_{\la^1, \ldots, \la^{h-1}} 
	\prod_{i=1}^{h}
s^{\delta}_i (Z^i) 
	&=  
\prod_{\substack{1\leq i<j \leq h  \\ \delta_i > \delta_j   } } \Psi(Z^i,Z^j) 
\times
	\sum_{ \gamma}   s_{\la^0/\gamma} (\mathbf{Z}^\delta_{-}) 
	s_{\la^{h}/\gamma}(\mathbf{Z}^\delta_+); 
	\label{eq:MainOpen}\\
	\sum_{\la^0, \ldots, \la^{h}} z^{|\la^h|}
	\prod_{i=1}^{h}
s^{\delta}_i (Z^i) 
	&=  
\prod_{\substack{1\leq i<j \leq h  \\ \delta_i > \delta_j   } } \Psi(Z^i,Z^j) 
\times
	\Phi(\mathbf{Z}^\delta_-)  \prod_{k\geq 1} \frac{\Phi(z^k\mathbf{Z})}{1-z^k};
	\label{eq:MainComplete}\\
	\sum_{\substack{\la^0, \ldots, \la^h\\ \la^0=\la^h}} 
z^{|\la^h|} 
	\prod_{i=1}^{h}
s^{\delta}_i (Z^i)  
	&=\prod_{\substack{1\leq i<j \leq h  \\ \delta_i > \delta_j   } } \Psi(Z^i,Z^j)
\times
	\prod_{k\geq 1} \frac{\Psi(z^k\mathbf{Z}^\delta_-, \mathbf{Z}^\delta_+)}{1-z^k}. 
\label{eq:MainCylinder}
\end{align}
\end{thm}

\goodbreak

\medskip

Actually, Theorem \ref{th:MainZ} is equivalent to the following Theorem \ref{th:SchurSum:sn},
since with 
$X^{i-1}=\emptyset$ and $Y^{i-1}=Z^{i}$ if $\delta_i=1$; 
$Y^{i-1}=\emptyset$ and $X^{i-1}=Z^i$ if $\delta_i=-1$,  
we recover
Theorem \ref{th:MainZ}.
Therefore we just need to prove Theorem \ref{th:SchurSum:sn}.

\begin{thm}\label{th:SchurSum:sn}
Suppose that $X^0, X^1, \ldots, X^{h-1}$ and $Y^0, Y^1, \ldots, Y^{h-1}$ are 
	$2h$ alphabets. Let $\mathbf{X}=\sum_{i=0}^{h-1}X^i$ and $\mathbf{Y}=\sum_{i=0}^{h-1}Y^i$ be the union of $X^0, X^1, \ldots, X^{h-1}$ and $Y^0, Y^1, \ldots, Y^{h-1}$ respectively.  Then we have  

\begin{align} 
	&\sum_{\la^1, \ldots, \la^{h-1}} 
	\sum_{\mu^0, \ldots, \mu^{h-1}} 
	\prod_{i=0}^{h-1}
s_{\la^i/\mu^i} (X^i)
	s_{\la^{i+1}/\mu^i}(Y^i) 
\label{eq:SchurSum:sn:open}
\\
&\qquad =  
\prod_{0\leq i<j \leq {h-1}}\Psi(Y^i,X^j)
\sum_{ \gamma}   s_{\la^0/\gamma} (\mathbf{X}) s_{\la^{h}/\gamma}(\mathbf{Y}).  \nonumber\\
&\sum_{\la^0, \ldots, \la^h} 
	\sum_{\mu^0, \ldots, \mu^{h-1}} 
z^{|\la^h|} 
	\prod_{i=0}^{h-1}
s_{\la^i/\mu^i} (X^i)
	s_{\la^{i+1}/\mu^i}(Y^i) 
\label{eq:SchurSum:sn}
\\
&\qquad =  
\prod_{0\leq i<j \leq {h-1}}\Psi(Y^i,X^j) \times
\Phi(\mathbf{X})  \prod_{k\geq 1} 
	\frac{\Phi(z^k(\mathbf{X}+\mathbf{Y}))}{1-z^k}.\nonumber\\
	&  \sum_{\substack{\la^0, \ldots, \la^h\\ \la^0=\la^h}} 
	\sum_{\mu^0, \ldots, \mu^{h-1}} 
z^{|\la^h|} 
	\prod_{i=0}^{h-1}
s_{\la^i/\mu^i} (X^i)
	s_{\la^{i+1}/\mu^i}(Y^i) 
\label{eq:SchurSum:sn:cylinder}
\\
&\qquad =  
\prod_{0\leq i<j \leq {h-1}}\Psi(Y^i,X^j) \times
\prod_{k\geq 1} \frac{\Psi(z^k\mathbf{X}, \mathbf{Y})}{1-z^k}. \nonumber
\end{align}
\end{thm}

\medskip

To give the proof of Theorem \ref{th:SchurSum:sn}, let us recall the following two formulas stated in 
Macdonald's book \cite{Macdonald1995} (see p. 93, ex. 26(1) and ex. 27(3)).
\begin{align}	
	\sum_{\rho} s_{\rho/\la}(X) s_{\rho/\mu}(Y)
	&= \Psi(X,Y) \sum_{\rho} s_{\la/\rho}(Y) s_{\mu/\rho}(X), 
	\label{eq:SchurSum:Macdonald:p93A} \\
	\sum_\rho s_{\rho/\nu} &= \Phi(X) \sum_\rho s_{\nu/\rho}.
	\label{eq:SchurSum:Macdonald:p93B}
\end{align}

First we use \eqref{eq:SchurSum:Macdonald:p93B} to prove some lemmas.

\begin{lem}\label{th:SchurSum:s1}
We have
\begin{align}
\sum_{\mu, \tau} z^{|\mu|} s_{\mu/\tau}(X) &=
\prod_{k\geq 1} \frac{\Phi(z^kX)}{1-z^k}. \label{eq:SchurSum:s1}
\end{align}
\end{lem}

\begin{proof}
Let $F(X)$ be the left-hand side of \eqref{eq:SchurSum:s1}. Then, by 
\eqref{eq:SchurSum:Macdonald:p93B} we have
\begin{align*}
F(X)&=\sum_{\mu, \tau} z^{|\mu|} s_{\mu/\tau}(X) =\sum_{\tau} z^{|\tau|} \sum_\mu  s_{\mu/\tau}(zX)\\
&= \Phi(zX) \sum_{\tau} z^{|\tau|}\sum_\rho  s_{\tau/\rho}(zX)= \Phi(zX) F(zX).
\end{align*}
Hence, we obtain 
$$
\sum_{\mu, \tau} z^{|\mu|} s_{\mu/\tau}(X) =
\prod_{k\geq 1} \Phi(z^kX) \times  F(\emptyset).
$$
Since
$$
F(\emptyset)=\sum_{\mu, \tau} z^{|\mu|} s_{\mu/\tau}(\emptyset) = \sum_{\mu} z^{|\mu|} = \prod_{k\geq 1} \frac{1}{1-z^k},
$$
then \eqref{eq:SchurSum:s1} is proved.
\end{proof}

\begin{lem}\label{th:SchurSum:s2}
	We have
	\begin{equation}\label{eq:SchurSum:s2}
\sum_{\lambda, \mu, \gamma} z^{|\mu|}   s_{\mu/\gamma} (X) s_{\lambda/\gamma}(Y)
=  \Phi(Y)  \prod_{k\geq 1} \frac{\Phi(z^k(X+Y))}{1-z^k}.
\end{equation}
\end{lem}

\begin{proof}
	By \eqref{eq:SchurSum:sn:open} and Lemma \ref{th:SchurSum:s1}	we have
\begin{align*}
&\sum_{\lambda, \mu} z^{|\mu|} \sum_\gamma 
	 s_{\mu/\gamma} (X) s_{\lambda/\gamma}(Y)\\
= & \sum_{\mu} z^{|\mu|} \sum_\gamma  s_{\mu/\gamma} (X)  \sum_\lambda s_{\lambda/\gamma}(Y)\\
	= & \Phi(Y) \sum_{\mu} z^{|\mu|} \sum_\gamma  s_{\mu/\gamma} (X)  \sum_\tau s_{\gamma/\tau}(Y)\\
	= & \Phi(Y) \sum_{\mu} z^{|\mu|} \sum_\tau s_{\mu/\tau}(X+Y)\\
	= & \Phi(Y)  \prod_{k\geq 1} \frac{ \Phi(z^k(X+Y)) }{1-z^k} .\qedhere
\end{align*}
\end{proof}

\textbf{Remark.} Formula \eqref{eq:SchurSum:s2} is similar to the following formula
stated in Macdonald's book \cite{Macdonald1995} (see p. 94, ex. 28(a)), which has two free partitions $\la$ and $\gamma$:
\begin{equation}\label{eq:SchurSum:Macdonald:p94A}
	\sum_{\la, \gamma} z^{|\la|} s_{\la/\gamma}(X) s_{\la/\gamma}(Y)
=  \prod_{k\geq 1} \frac{\Psi(z^kX, Y)}{1-z^k}.
\end{equation}
\medskip
Now we are ready to give the proof of Theorem \ref{th:SchurSum:sn}.

\begin{proof}[Proof of the Theorem \ref{th:SchurSum:sn}]
Let $F(X^0, X^1, \ldots, X^{h-1},Y^0, Y^1, \ldots, Y^{h-1}) $ be the left-hand side of \eqref{eq:SchurSum:sn:open}.	By \eqref{eq:SchurSum:Macdonald:p93A}
	we have
\begin{align*}
	&F(X^0, X^1, \ldots, X^{h-1},Y^0, Y^1, \ldots, Y^{h-1}) 
\\
	&=\sum_{\la^1, \ldots, \la^{h-2}} 
	\sum_{\mu^0, \ldots, \mu^{h-1}} 
	\prod_{i=0}^{h-3}
s_{\la^i/\mu^i} (X^i)
	s_{\la^{i+1}/\mu^i}(Y^i) 
\\& 
	\quad \times
s_{\la^{h-2}/\mu^{h-2}}(X^{h-2}) 
s_{\la^{h}/\mu^{h-1}} (Y^{h-1}) 
\sum_{\la^{h-1}}
s_{\la^{h-1}/\mu^{h-2}} (Y^{h-2}) 
s_{\la^{h-1}/\mu^{h-1}}(X^{h-1}) 
\\
&
=\Psi(Y^{h-2},X^{h-1})
	\sum_{\la^1, \ldots, \la^{h-2}} 
	\sum_{\mu^0, \ldots, \mu^{h-1}} 
	\prod_{i=0}^{h-3}
s_{\la^i/\mu^i} (X^i)
	s_{\la^{i+1}/\mu^i}(Y^i) 
\\& \quad\times
s_{\la^{h-2}/\mu^{h-2}}(X^{h-2}) 
s_{\la^{h}/\mu^{h-1}} (Y^{h-1}) 
\sum_{\la^{h-1}}
s_{\mu^{h-2}/\la^{h-1}} (X^{h-1}) 
s_{\mu^{h-1}/\la^{h-1}} (Y^{h-2})
\\
&
=\Psi(Y^{h-2},X^{h-1})
	\sum_{\la^1, \ldots, \la^{h-1}} 
	\sum_{\mu^0, \ldots, \mu^{h-3}} 
	\prod_{i=0}^{h-3}
s_{\la^i/\mu^i} (X^i)
	s_{\la^{i+1}/\mu^i}(Y^i) 
\\& \ \times
\sum_{\mu^{h-2}}
s_{\la^{h-2}/\mu^{h-2}}(X^{h-2}) 
s_{\mu^{h-2}/\la^{h-1}} (X^{h-1}) 
\sum_{\mu^{h-1}}
s_{\la^{h}/\mu^{h-1}} (Y^{h-1}) 
s_{\mu^{h-1}/\la^{h-1}} (Y^{h-2})
\\
&
=\Psi(Y^{h-2},X^{h-1})
	\sum_{\la^1, \ldots, \la^{h-1}} 
	\sum_{\mu^0, \ldots, \mu^{h-3}} 
	\prod_{i=0}^{h-3}
s_{\la^i/\mu^i} (X^i)
	s_{\la^{i+1}/\mu^i}(Y^i) 
\\& \quad\times
s_{\la^{h-2}/\la^{h-1}}(X^{h-2}+X^{h-1}) 
s_{\la^{h}/\la^{h-1}} (Y^{h-2}+Y^{h-1})
\\
&
=\Psi(Y^{h-2},X^{h-1})
	\sum_{\la^1, \ldots, \la^{h-2}} 
	\sum_{\mu^0, \ldots, \mu^{h-2}} 
	\prod_{i=0}^{h-3}
s_{\la^i/\mu^i} (X^i)
	s_{\la^{i+1}/\mu^i}(Y^i) 
\\& \quad\times
s_{\la^{h-2}/\mu^{h-2}}(X^{h-2}+X^{h-1}) 
s_{\la^{h}/\mu^{h-2}} (Y^{h-2}+Y^{h-1})
\\
&=
\Psi(Y^{h-2},X^{h-1})
F(X^0, \ldots, X^{h-3},X^{h-2}+X^{h-1},Y^0, \ldots, Y^{h-3},Y^{h-2}+Y^{h-1})
\\
&=\cdots
\\
&=\Bigl(\prod_{0\leq i<j \leq {h-1}}\Psi(Y^i,X^j)\Bigr)
F(X^0+ \ldots+X^{h-1},Y^0+ \ldots+ Y^{h-1})
\\
&=\Bigl(\prod_{0\leq i<j \leq {h-1}}\Psi(Y^i,X^j)\Bigr)
\sum_{ \gamma}    s_{\la^0/\gamma} (X^0+ \ldots+X^{h-1}) s_{\la^{h}/\gamma}(Y^0+ \ldots+ Y^{h-1}).
\end{align*}

Therefore, identity \eqref{eq:SchurSum:sn:open} is true. Then
identities \eqref{eq:SchurSum:sn} and \eqref{eq:SchurSum:sn:cylinder} hold by \eqref{eq:SchurSum:sn:open}, \eqref{eq:SchurSum:Macdonald:p94A} and Lemma \ref{th:SchurSum:s2}.
\end{proof}

% >>>

\section{Definitions for skew doubled shifted plane partitions}\label{sec:SkRPP} % <<<
% Theorem \ref{th:MainZ} is powerful for deriving generating functions of plane partitions.  
In this section we give the definition and the trace generating function of skew doubled shifted plane partitions.
Each profile $\delta$ is associated with a connected area  
$\Delta:=\Delta(\delta)$ 
of the quarter plane~$\Lambda$
in a unique manner.
For a given profile 
$$
\delta=1^{a_0}(-1)^{b_1}1^{a_1}(-1)^{b_2}\ldots 1^{a_{r-1}}(-1)^{b_r}
$$ 
with $a_0,b_r\geq 0$, $a_i,b_i\geq 1$ for $1\leq i \leq r-1$, let 
\begin{align*}
\Delta_1&=\bigcup_{i=1}^{r-1}\{ (c,d)\in \Lambda: \sum_{j=1}^{r-i-1}a_{r-j}  \leq c \leq \sum_{j=1}^{r-i}a_{r-j}, 1 \leq d\leq \sum_{j=1}^i b_i \},\\
\Delta_2&=\{(c,d)\in \Lambda: c-d > \sum_{i=0}^{r-1}a_i \}, \\
\Delta_3&=\{(c,d)\in \Lambda: d-c > \sum_{i=1}^{r}b_i \}.
\end{align*}
The connected area $\Delta$ is defined to be
$\Delta :=\Lambda \setminus (\Delta_1\cup \Delta_2\cup\Delta_3)$.
For example, with the profile $\delta=(1,-1,-1,1,-1,1,-1,1)$, the four
areas $\Delta_1, \Delta_2, \Delta_3, \Delta$ are illustrated in Fig. 3.

\medskip

Let $\la$ and $\mu$ be two integer partitions.
	We write $\la \succ \mu$ or $\mu \prec \la$ if $\la / \mu$ is a horizontal strip (see \cite{Langer2013A,Langer2013B,Macdonald1995, Okounkov2003, Okoun2006,  StanleyEC2}).

\begin{defi} \label{skew-double-shifted PP}
Let $\delta=(\delta_i)_{1\leq i\leq h}$ be a profile.
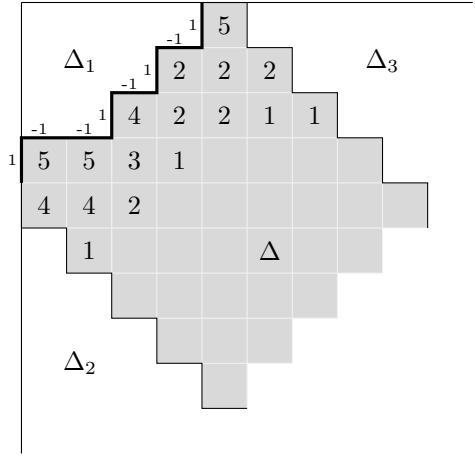
A \emph{skew doubled shifted plane partition $(\PP{DSPP})$} 
	with profile $\delta$ is a filling $\omega=(\omega_{i,j})$ of $\Delta(\delta)$ with nonnegative integers such that 
the size $|\omega|=\sum_{(i,j)} \omega_{i,j}$ is finite, and
the rows and columns are weakly decreasing, i.e.,
$$
\omega_{i,j}\geq \omega_{i,j+1},\quad \omega_{i,j}\geq \omega_{i+1,j}
$$
	whenever these numbers are well-defined. 
\end{defi}

\goodbreak

The set of all \PP{DSPP} with profile $\delta$ is denoted by 
$\PP{DSPP}_\delta$. Recall that the Schur process for plane partitions was first introduced by Okounkov and Reshetikhin \cite{Okounkov2003}  (see also \cite{Langer2013A,Langer2013B, Okoun2006}); the main idea was to read the plane partitions along the diagonals. When reading the \PP{DSPP} $\omega$ with profile $\delta$ along the diagonals from left to right, we obtain a sequence of integer partitions
$(\la^0, \la^1, \ldots, \la^h)$
such that $\la^{i-1}\prec \la^i $ (resp. $\la^{i-1} \succ \la^i$)
if $\delta_i = 1$  (resp.  $\delta_i = -1$),  
and $|\omega|=\sum_{i=0}^h |\la^i|$. For simplicity, we identify the skew doubled shifted plane partition $\omega$ and the sequence of integer partitions by writing
$$
\omega=
(\la^0, \la^1, \ldots, \la^h).$$
The {\it (diagonal) width} of the skew doubled shifted plane partition $\omega$ is defined to be $h+1$.

For example, with the \PP{DSPP} $\omega$ given in Fig. 3,
we obtain a sequence of partitions: $(4,1)\prec (5,4) \succ (5,2) \succ (3) \prec (4,1)\succ (2) \prec  (2,2)\succ(2,1)\prec(5,2,1)$. Hence, $\omega=\bigl((4,1), (5,4), (5,2), (3), (4,1), (2),  (2,2),(2,1),(5,2,1)\bigr)$ is an 
	\PP{DSPP} of width 9 with profile $\delta=(1,-1,-1,1,-1,1,-1,1)$.

$$
\begin{tikzpicture}[scale=0.6]
\begin{scope}[xshift=0, yshift=0]
\fill [gray!30](4.0000,3.0000)--(3.0000,3.0000)--(3.0000,4.0000)--(2.0000,4.0000)--(2.0000,5.0000)--(1.0000,5.0000)--(1.0000,6.0000)--(0.0000,6.0000)--(0.0000,7.0000)--(-1.0000,7.0000)--(-1.0000,8.0000)--(-1.0000,9.0000)--(0.0000,9.0000)--(1.0000,9.0000)--(1.0000,10.0000)--(2.0000,10.0000)--(2.0000,11.0000)--(3.0000,11.0000)--(3.0000,12.0000)--(4.0000,12.0000)--(4.0000,11.0000)--(5.0000,11.0000)--(5.0000,10.0000)--(6.0000,10.0000)--(6.0000,9.0000)--(7.0000,9.0000)--(7.0000,8.0000)--(8.0000,8.0000)--(8.0000,7.0000)--(7.0000,7.0000)--(7.0000,6.0000)--(6.0000,6.0000)--(6.0000,5.0000)--(5.0000,5.0000)--(5.0000,4.0000)--(4.0000,4.0000)--(4.0000,3.0000);
\draw [gray!10](-1.0000,7.0000)--(-1.0000,9.0000);
\draw [gray!10](0.0000,6.0000)--(0.0000,9.0000);
\draw [gray!10](1.0000,5.0000)--(1.0000,10.0000);
\draw [gray!10](2.0000,4.0000)--(2.0000,11.0000);
\draw [gray!10](3.0000,3.0000)--(3.0000,12.0000);
\draw [gray!10](4.0000,3.0000)--(4.0000,12.0000);
\draw [gray!10](5.0000,4.0000)--(5.0000,11.0000);
\draw [gray!10](6.0000,5.0000)--(6.0000,10.0000);
\draw [gray!10](7.0000,6.0000)--(7.0000,9.0000);
\draw [gray!10](3.0000,3.0000)--(4.0000,3.0000);
\draw [gray!10](2.0000,4.0000)--(5.0000,4.0000);
\draw [gray!10](1.0000,5.0000)--(6.0000,5.0000);
\draw [gray!10](0.0000,6.0000)--(7.0000,6.0000);
\draw [gray!10](-1.0000,7.0000)--(8.0000,7.0000);
\draw [gray!10](-1.0000,8.0000)--(8.0000,8.0000);
\draw [gray!10](-1.0000,9.0000)--(7.0000,9.0000);
\draw [gray!10](1.0000,10.0000)--(6.0000,10.0000);
\draw [gray!10](2.0000,11.0000)--(5.0000,11.0000);
\draw [black](4.0000,3.0000)--(3.0000,3.0000)--(3.0000,4.0000)--(2.0000,4.0000)--(2.0000,5.0000)--(1.0000,5.0000)--(1.0000,6.0000)--(0.0000,6.0000)--(0.0000,7.0000)--(-1.0000,7.0000)--(-1.0000,8.0000)--(-1.0000,9.0000)--(0.0000,9.0000)--(1.0000,9.0000)--(1.0000,10.0000)--(2.0000,10.0000)--(2.0000,11.0000)--(3.0000,11.0000)--(3.0000,12.0000)--(4.0000,12.0000)--(4.0000,11.0000)--(5.0000,11.0000)--(5.0000,10.0000)--(6.0000,10.0000)--(6.0000,9.0000)--(7.0000,9.0000)--(7.0000,8.0000)--(8.0000,8.0000)--(8.0000,7.0000);
\draw (-1.0000,12.0000)--(-1.0000,2.0000);
\draw (-1.0000,12.0000)--(9.0000,12.0000);
\draw (-0.5000, 11.5000) node [] {}; 
\draw (0.5000, 11.5000) node [] {}; 
\draw (1.5000, 11.5000) node [] {}; 
\draw (2.5000, 11.5000) node [] {}; 
\draw (3.5000, 11.5000) node [] {5}; 
\draw (-0.5000, 10.5000) node [] {}; 
\draw (0.5000, 10.5000) node [] {}; 
\draw (1.5000, 10.5000) node [] {}; 
\draw (2.5000, 10.5000) node [] {2}; 
\draw (3.5000, 10.5000) node [] {2}; 
\draw (4.5000, 10.5000) node [] {2}; 
\draw (-0.5000, 9.5000) node [] {}; 
\draw (0.5000, 9.5000) node [] {}; 
\draw (1.5000, 9.5000) node [] {4}; 
\draw (2.5000, 9.5000) node [] {2}; 
\draw (3.5000, 9.5000) node [] {2}; 
\draw (4.5000, 9.5000) node [] {1}; 
\draw (5.5000, 9.5000) node [] {1}; 
\draw (-0.5000, 8.5000) node [] {5}; 
\draw (0.5000, 8.5000) node [] {5}; 
\draw (1.5000, 8.5000) node [] {3}; 
\draw (2.5000, 8.5000) node [] {1}; 
\draw (-0.5000, 7.5000) node [] {4}; 
\draw (0.5000, 7.5000) node [] {4}; 
\draw (1.5000, 7.5000) node [] {2}; 
\draw (-0.5000, 6.5000) node [] {}; 
\draw (0.5000, 6.5000) node [] {1}; 
\draw (-0.6000, 11.2000) node [] {}; 
\draw (0.4000, 11.2000) node [] {}; 
\draw (1.4000, 11.2000) node [] {}; 
\draw (2.4000, 11.2000) node [] {{\tiny -1}}; 
\draw (-0.6000, 10.2000) node [] {}; 
\draw (0.4000, 10.2000) node [] {}; 
\draw (1.4000, 10.2000) node [] {{\tiny -1}}; 
\draw (-0.6000, 9.2000) node [] {{\tiny -1}}; 
\draw (0.4000, 9.2000) node [] {{\tiny -1}}; 
\draw (-1.2000, 11.5000) node [] {}; 
\draw (-0.2000, 11.5000) node [] {}; 
\draw (0.8000, 11.5000) node [] {}; 
\draw (1.8000, 11.5000) node [] {}; 
\draw (2.8000, 11.5000) node [] {{\tiny 1}}; 
\draw (-1.2000, 10.5000) node [] {}; 
\draw (-0.2000, 10.5000) node [] {}; 
\draw (0.8000, 10.5000) node [] {}; 
\draw (1.8000, 10.5000) node [] {{\tiny 1}}; 
\draw (-1.2000, 9.5000) node [] {}; 
\draw (-0.2000, 9.5000) node [] {}; 
\draw (0.8000, 9.5000) node [] {{\tiny 1}}; 
\draw (-1.2000, 8.5000) node [] {{\tiny 1}}; 
\draw [very thick](-1.0000,8.0000)--(-1.0000,9.0000)--(0.0000,9.0000)--(1.0000,9.0000)--(1.0000,10.0000)--(2.0000,10.0000)--(2.0000,11.0000)--(3.0000,11.0000)--(3.0000,12.0000);
%\draw (3,1) node [] {I: A skew ribbon plane partition};
\draw (0+0.3,11-0.3) node [] {$\Delta_1$};
\draw (0+0.3,4) node [] {$\Delta_2$};
\draw (7,11-0.3) node [] {$\Delta_3$};
\draw (4.5, 6.5) node [] {$\Delta$};
\end{scope}
\end{tikzpicture}
$$
\centerline{Fig. 3. A skew doubled shifted plane partition.}
\medskip

For a sequence of parameters $u_i\ (i\geq 0),$ write 
	$U_j=u_0u_1\cdots u_{j-1}$ ($j\geq 0$). 
Let $\PP{DSPP}_\delta(\la^0, \la^h)$ denote the set of 
	the skew doubled shifted plane partitions 
	$\omega=(\la^0,\la^1, \ldots, \la^h)$
	starting from $\la^0$ and ending at $\la^h$ with profile $\delta$. 
Let $Z^i=\{ U_i^{-\delta_i} \}$ in 
Theorem \ref{th:MainZ}, 
we obtain the following trace generating functions for skew doubled shifted plane partitions.
\begin{thm}\label{th:MainU}
Let $\delta=(\delta_i)_{1\leq i\leq h}$ be a profile.  We have
\begin{align}
	\sum_{\omega \in \PP{DSPP}_\delta(\la^0, \la^h)} \prod_{i=0}^h u_i^{|\la^i|}
&=U_{h+1}^{|\la^h|}
\prod_{\substack{1\leq i<j \leq h  \\ \delta_i > \delta_j   } } 
\frac{1}{1-U_i^{-1}U_j}  
\label{eq:MainUOpen}
\\
& \kern -15mm \times
\sum_{ \gamma}   s_{\la^0/\gamma} (\{U_i:\delta_i=-1    \})  s_{\la^{h}/\gamma}(\{U_i^{-1}:\delta_i=1\}); \nonumber \\
\sum_{\omega \in \PP{DSPP}_\delta} \prod_{i=0}^h u_i^{|\la^i|} 
&= \prod_{\substack{1\leq i<j \leq h  \\ \delta_i > \delta_j   } } 
\frac{1}{1-U_i^{-1}U_j}  \label{eq:MainUComplete}
\\
& \kern -15mm \times \nonumber
\Phi(\{U_i:\delta_i=-1    \})  \prod_{k\geq 1} 
\frac{\Phi(\{U_i^{-\delta_i}U_{h+1}^k: 1\leq i \leq h  \} )}{1-U_{h+1}^k}; \\
	\sum_{\omega \in \PP{DSPP}_\delta(\la^0= \la^h)} \prod_{i=0}^h u_i^{|\la^i|} 
&= \prod_{\substack{1\leq i<j \leq h  \\ \delta_i > \delta_j   } } 
\frac{1}{1-U_i^{-1}U_j}  \label{eq:MainUCylinder} 
\\ 
& \kern -15mm\times \nonumber
\prod_{k\geq 1} \frac{\Psi(\{U_iU_{h+1}^k:\delta_i=-1    \}, \{U_j^{-1}:\delta_j=1    \})}{1-U_{h+1}^k}.
\end{align}
\end{thm}

The above theorem implies many classical results on various defective plane partitions, including the trace generating function of  ordinary plane partitions $(\PP{PP})$ (Stanley \cite{Stanley1973}), and the generating functions of symmetric plane partitions $(\PP{SPP})$ (Andrews \cite{Andrews1978PP1}, Macdonald~\cite{Macdonald1995}), skew plane partitions (Sagan \cite{Sagan1993}) and skew shifted plane partitions (Sagan \cite{Sagan1993}).

% \section{Alternating summation formulas of skew Schur functions}
%  \label{sec:alternating-summation-formula-Schur-functions} 
% >>>

\section{Useful theorems for asymptotic formulas}\label{sec:Asym} % <<<
\medskip

Inspired by the works of Dewar, Murty and Kot\v{e}\v{s}ovec \cite{DewarMurty2013, Kotesovec2015}, we have established some useful theorems for
asymptotic formulas in \cite{HanXiong2017}. The purpose of this section is to restate these results. 
Define
\begin{equation}\label{def:psi}
\psi_{n}(v, r, b; p) := v \sqrt{\frac{p(1-p)}{2\pi}} \frac{r^{b+(1-p)/2}}{n^{b+1-p/2}} \exp({n^p}{r^{1-p}})
\end{equation}
for $n\in \mathbb{N},$ $v,b\in \mathbb{R},$ $r>0,$ $0<p<1$.

\begin{thm}[\cite{HanXiong2017}]\label{th:asy_main2}
Let $t_1$ and $t_2$ be given positive integers with $\gcd(t_1,t_2)=1$. 
	Suppose that
	\begin{equation*}
		F_1(q) = \sum_{n=0}^\infty a_{t_1n} q^{t_1n} \qquad \text{and}\qquad
		F_2(q) = \sum_{n=0}^\infty c_{t_2n} q^{t_2n}
	\end{equation*}
	are two power series such that their coefficients satisfy the 
	asymptotic formulas
	\begin{align*}
		a_{t_1n} &\ \sim\  t_1\psi_{t_1n}(v_1, r_1, b_1; p), \\
		c_{t_2n} &\ \sim\  t_2\psi_{t_2n}(v_2, r_2, b_2; p),
	\end{align*}
	with $0<p<1,$ $r_1>0,r_2>0,$ $v_1,b_1,v_2, b_2\in \mathbb{R}$. 
Then, the coefficients $d_n$ in the product
	$$F_1(q)F_2(q) = \sum_{n=0}^\infty d_n q^{n}$$ 
	satisfy the following asymptotic formula
	\begin{align}\label{eq:asy_d_n}
		d_n \sim 
		\psi_{n}(v_1v_2, r_1+r_2, b_1+b_2 ; p).
	\end{align}
\end{thm}

\begin{thm}[\cite{HanXiong2017}]\label{th:Asy_Ribbon:multi}
Let $m$ be a positive integer. 
Suppose that $x_i$ and 
	$y_i$ $(1\leq i\leq m)$ are 
positive integers such that 
$\gcd(x_1, x_2, \ldots, x_m, y_1, y_2, \ldots, y_m)=1$.
Then, the coefficients $d_n$ in the following infinite product
$$
\prod_{i=1}^m \prod_{k\geq 0}\frac {1}{1-q^{x_ik+y_i}}=\sum_{n=0}^\infty d_{n} q^{n}
$$
have the following asymptotic formula
\begin{align}\label{eq:asy_product_2}
d_{n} \ \sim  \ 
	 v {\frac{1}{2\sqrt{2\pi}}} 
	\frac{r^{b+1/4}}{n^{b+3/4}} \exp(\sqrt{nr}),
\end{align}
where
$$
	v=\prod_{i=1}^m \frac{\Gamma(y_i/x_i)}{\sqrt{x_i\pi}} (\frac {x_i}{2})^{y_i/x_i}, \qquad
	r=\sum_{i=1}^m\frac {2\pi^2}{3x_i}, \qquad 
	b=\sum_{i=1}^m(\frac{y_i}{2x_i} - \frac 14).
$$   
\end{thm}

% >>>

\section{Formulas for skew doubled shifted plane partitions} \label{sec:gfRPP}% <<<

Let $u_i=z$ for $i\geq 0$ in  \eqref{eq:MainUComplete}. We then derive
the generating function for $\PP{DSPP}_\delta$:
\begin{align*}
& \sum_{\omega \in \PP{DSPP}_\delta}  z^{|\omega|}=
\prod_{\substack{1\leq j<i \leq h  \\ \delta_i < \delta_j   } } \frac{1}{1-z^{i-j}}
\times \Phi(\{z^i:\delta_i=-1    \}) 
\\
&\kern 28mm\times
\prod_{k\geq 1}    \frac{\Phi(\{z^{(h+1)k+i}: \delta_i=-1  \}  + \{z^{(h+1)k-j}:\delta_j=1    \} )}{1-z^{(h+1)k}}.
\end{align*}
The right-hand side of the above identity can be further simplified. 
For each profile $\delta=(\delta_i)_{1\leq i\leq m-1}$,
we define the following {multisets}  as 
\begin{align*}
	W_1(\delta)&:=\{m\} \cup \{ i \mid  \delta_i=-1\} \cup \{ m-i\mid  \delta_{i}=1 \}; \\
	W_2(\delta) 
	&:=\{\quad i+j\quad \mid  1\leq i<j \leq m-1,\ \delta_i=\delta_{j}=-1 \}   \\
	&\cup \{ 2m-i-j\mid 1\leq i<j \leq m-1,\ \delta_i=\delta_{j}=1 \}  \\
	&\cup \{ 2m+i-j\mid 1\leq i<j \leq m-1,\ \delta_i<\delta_{j} \}  \\
  &\cup \{\quad j-i\quad \mid    1\leq i<j \leq m-1,\ \delta_i>\delta_{j} \}. 
\end{align*}

\begin{thm}\label{th:Skew_Double_Shifted_PP}
The generating function for the
skew doubled shifted plane partitions  with profile $\delta=(\delta_i)_{1\leq i\leq m-1}$
is
\begin{equation*}
\sum_{\omega \in \PP{DSPP}_\delta}  z^{|\omega|}
=  
	\prod_{k\geq 0}  
	\left(\prod_{t\in W_1(\delta)} \frac{1}{1-z^{mk+t}}\right) 
	\left(\prod_{t\in W_2(\delta)} \frac{1}{1-z^{2mk+t}}\right).
\end{equation*}
\end{thm}

By using the above generating function,
our next theorem gives asymptotic formulas for the number of skew doubled shifted plane 
partitions. For simplicity, write
$$
\epsilon(\delta):=\, -\sum_{i} \delta_i\, i \,=\sum_{\delta_i=-1} i - \sum_{\delta_j=1} j.
$$

\begin{thm}\label{th:Skew_Double_Shifted_PP_Asy}
Let $m\geq 2$ be a positive integer and
$\delta=(\delta_j)_{1\leq j\leq m-1}$ be a profile of length $m-1$.
	We denote by $\PP{DSPP}_\delta(n)$ the number of 
	skew doubled shifted plane partitions $\omega$ with profile $\delta$ and size $n$. Then,
$$
	\PP{DSPP}_\delta(n) \ \sim \  
	C_1(\delta)C_2(m) \times  \frac{1}{n} 
	\exp\Bigl(\pi\sqrt{\frac{(m^2+m+2)n}{6m}} \, \Bigr),
$$
where $C_1(\delta)$ and $C_2(m)$ are two constants with respect to $n$:
%It depends only on the height of the SkRPP, not the skew part. More precisely, 
\begin{align*}
C_1(\delta)&=
	2^{-\frac{\epsilon(\delta)}{m} 
  -|\delta|_1}\times
	\prod_{t\in W_1(\delta)} \Gamma(\frac{t}{m})\prod_{t\in W_2(\delta)} \Gamma(\frac{t}{2m}),
\\
C_2(m)
	&=    \Bigl( 2^{ m^2-3m+14 } {{\pi}^{{m^2-m}}} \Bigr)^{-\frac 14}
	\times \sqrt{\frac{m^2+m+2}{3}}.
\end{align*}
\end{thm}

\begin{proof}
By the definitions of $W_1(\delta)$ and $W_2(\delta)$
we have
	\begin{align*}
		\#W_1(\delta) &= m;\\
		\#W_2(\delta)&= \binom{m-1}{2};\\
		\sum_{t\in W_1(\delta)} t &= m(|\delta|_1+1)+ \epsilon(\delta); \\
		\sum_{t\in W_2(\delta)} t &= (m-2) \epsilon(\delta)
+ 2m\binom{|\delta|_1}{2}
+2m \sum_{\substack{1\leq i<j \leq m-1  \\ \delta_i < \delta_j   } } 1.\\
\end{align*}
Hence,
	\begin{equation}\label{eq:W1W2sum1}
\sum_{t\in W_1(\delta)} \frac{1}{m} + \sum_{t\in W_2(\delta)} \frac{1}{2m} 
		= \frac{m}{m} + \frac 12 \binom{m-1}{2} = {\frac{m^2+m+2}{4m}}.
	\end{equation}
Furthermore,
\begin{equation*}
\sum_{t\in W_1(\delta)} \frac{t}{m}  + \sum_{t\in W_2(\delta)} \frac{t}{2m} 
=
	1+|\delta|_1+\binom{|\delta|_1}{2}+ \sum_{\substack{1\leq i<j \leq m-1  \\ \delta_i < \delta_j   } } 1 + \frac{\epsilon(\delta)}{2} .
\end{equation*}
If we exchange any two adjacent letters in $\delta$, 
the sum of the last two terms doesn't change, therefore it
	is equal to $\frac{1}{2} \binom{m}{2} - \binom{|\delta|_1+1}{2}$.
Then we obtain
\begin{equation*}
\sum_{t\in W_1(\delta)} \frac{t}{m}  + \sum_{t\in W_2(\delta)} \frac{t}{2m} 
%&=
%1+|\delta|_1+\binom{|\delta|_1}{2}+ \frac{1}{2} \binom{m}{2} - \binom{|\delta|_1+1}{2}  
=\frac{m^2-m+4}{4}
\end{equation*}
and
\begin{equation}\label{eq:W1W2sumt}
\sum_{t\in W_1(\delta)} (\frac{t}{2m} - \frac 14) + \sum_{t\in W_2(\delta)} (\frac{t}{4m} - \frac 14)   = \frac 14.
\end{equation}

By Theorems   \ref{th:Asy_Ribbon:multi} and \ref{th:Skew_Double_Shifted_PP}  the number of \PP{DSPP} with profile $\delta$ and size $n$ is asymptotic to
\begin{align*}
	 v {\frac{1}{2\sqrt{2\pi}}} 
	\frac{r^{b+1/4}}{n^{b+3/4}} \exp(\sqrt{nr}),
\end{align*}
where  
\begin{align*}
v&= \prod_{t\in W_1(\delta)} \left(\frac{\Gamma(t/m)}{\sqrt{m\pi}} (\frac {m}{2})^{t/m} \right)  
\prod_{t\in W_2(\delta)} \left(\frac{\Gamma(t/(2m))}{\sqrt{2m\pi}} ({m})^{t/(2m)} \right)
\\&=
		2^{-\frac{\epsilon(\delta)}{m}
  -|\delta|_1-1-\frac12 \binom{m-1}{2}} 
%\\&\times 
  { m^{1/{2} }} {{\pi}^{({-m^2+m-2})/{4}}}
	\\&\qquad\qquad	\times
	\prod_{t\in W_1(\delta)} \Gamma(\frac{t}{m})\prod_{t\in W_2(\delta)} \Gamma(\frac{t}{2m}),\\
%\end{align*}
%\begin{align*}
r&= \sum_{t\in W_1(\delta) } \frac {2\pi^2}{3m} + \sum_{t\in W_2(\delta) } \frac {2\pi^2}{6m}
= \frac{(m^2+m+2)\pi^2}{6m}, \\
%\end{align*}
%\begin{align*}
b&= \sum_{t\in W_1(\delta) }\left(\frac{t}{2m} - \frac 14\right)
+ \sum_{t\in W_2(\delta) } \left(\frac{t}{4m} - \frac 14\right)
=\frac{1}{4}.
\end{align*}
This achieves the proof.
\end{proof}

Theorems \ref{th:gfRPP} and \ref{th:AsymRPP} for doubled shifted plane partitions 
are specializations
of Theorems~\ref{th:Skew_Double_Shifted_PP}
and 
\ref{th:Skew_Double_Shifted_PP_Asy} for skew doubled shifted plane partitions
when taking the profile $\delta=(-1)^{m-1}$.

\begin{proof}[Proof of Theorems \ref{th:gfRPP} and \ref{th:AsymRPP}]
Take $\delta=(-1)^{m-1}$. Then we have
$$
W_1(\delta)=\{1,2,3,\ldots, m\},\quad W_2(\delta)=\{i+j: 1\leq i<j \leq m-1  \}.
$$
Therefore Theorem
\ref{th:Skew_Double_Shifted_PP} implies Theorem \ref{th:gfRPP}.
Since $\Gamma(z)\Gamma(1-z)=\frac{\pi}{\sin(z\pi)}$ and $\prod_{j=1}^{m-1}\sin(\frac{j\pi}{m})=\frac{m}{2^{m-1}}$, we have
$$\prod_{t\in W_1(\delta)} \Gamma(\frac{t}{m})
	=  
\sqrt{\frac{\pi^{m-1}}{\prod_{j=1}^{m-1}\sin(\frac{j\pi}{m})}}
= \sqrt{\frac{(2\pi)^{m-1}}{m}}
$$
and
$$\prod_{t\in W_2(\delta)} \Gamma(\frac{t}{2m})
=
	\pi^{(m-1)(m-2)/4}
	\left(\prod_{i=1}^{m-2}\,\prod_{j=i+1}^{m-i-1} {{\sin\left(\frac{i+j}{2m}\pi\right)}} \right)^{-1}.
$$
Since $\delta=(-1)^{m-1}$,
by Theorem \ref{th:Skew_Double_Shifted_PP_Asy}  we verify that
\begin{align*}
C_1(\delta)&=
	2^{-\frac{\epsilon(\delta)}{m} 
  -|\delta|_1}\times
	\prod_{t\in W_1(\delta)} \Gamma(\frac{t}{m})\prod_{t\in W_2(\delta)} \Gamma(\frac{t}{2m})
\\
	&= \frac { \pi^{(m^2-m)/4} }{\sqrt {m}}
	\left(\prod_{i=1}^{m-2}\,\prod_{j=i+1}^{m-i-1} {{\sin\left(\frac{i+j}{2m}\pi\right)}} \right)^{-1},
\end{align*}
	and that $C_1(\delta)C_2(m)$ is equal to $C(m)$ given in Theorem \ref{th:AsymRPP}. 
\end{proof}

For example, consider the three skew doubled shifted plane partitions (DSPPa)-(DSPPc) given in Fig. 2. Their profiles,
generating functions  and asymptotic formulas are respectively:

(a) Fig. 2, case DSPPa. $\delta=(1,1)$, $W_1(\delta)=\{3,2,1\}$, $W_2(\delta)=\{3\}$,
\begin{align*} 
	\sum_{\omega \in \PP{DSPPa}} z^{|\omega|} 
	&\ =\  \prod_{k\geq 0} \frac{1}{(1-z^{k+1})(1-z^{6k+3})},\\
\PP{DSPPa}(n) &\ \sim	 \ 
	\frac{\sqrt{7}}{24}\,\frac {\exp(\pi \frac{\sqrt {7n}}{3} )}{ n}.\label{eq:asym3a}
\end{align*}

(b) Fig. 2, case DSPPb.  $\delta=(1,-1)$,  $W_1(\delta)=\{3,2,2\}$, $W_2(\delta)=\{1\}$, 
\begin{align*}
\sum_{\omega \in \PP{DSPPb}} z^{|\omega|} 
	&\ =\ \prod_{k\geq 0}\frac{1}{(1-z^{3k+3})(1-z^{3k+2})^2(1-z^{6k+1})},\\
\PP{DSPPb}(n) 
	&\ \sim	 \ 
 \sqrt 2 {\alpha} \times {\frac {\sqrt 7}{24}}\,\frac {\exp(\pi \frac{\sqrt {7n}}{3} )}{ n},
\end{align*}
where
$$
\alpha= 2^{-\frac {11}{6}} \sqrt{3} \pi^{-\frac 32} \Gamma(\frac 23)^2 \Gamma(\frac 16)
=0.8908\cdots
$$

(c) Fig. 2, case DSPPc.  $\delta=(-1,1)$, $W_1(\delta)=\{3,1,1\}$, $W_2(\delta)=\{5\}$,
\begin{align*}
\sum_{\omega \in \PP{DSPPc}} z^{|\omega|} 
	&\ = \ \prod_{k\geq 0}\frac{1}{(1-z^{3k+3})(1-z^{3k+1})^2(1-z^{6k+5})},\\
\PP{DSPPc}(n) &\ \sim	 \ 
 \sqrt 2 {\alpha^{-1}} \times {\frac {\sqrt 7}{24}}\,\frac {\exp(\pi \frac{\sqrt {7n}}{3} )}{ n}.
\end{align*}

% >>>

\section{Formulas for symmetric cylindric partitions} \label{sec:gfSyCPP}% <<<
Cylindric partitions were introduced by Gessel and Krattenthaler~\cite{GesselKratt1997}. They obtained the generating function for cylindric partitions of some given shape that satisfy certain row bounds as some summation of determinants related to $q$-binomial coefficients. Later, Borodin gave an equivalent definition 
\cite{Borodin2007} and
obtained the generating function for cylindric partitions.
A \emph{cylindric partition $(\PP{CP})$} with profile $\delta$ is a \PP{DSPP} $$\omega=(\la^0,\la^1,\ldots,\la^{h-1},\la^h)$$ with profile $\delta$ such that $\la^0=\la^h$. The {\it size} of such partition is defined by $|\omega|=\sum_{i=0}^{h-1}|\la^i|$  (notice that $\la^h$ is not counted here).
The following generating function for cylindric partitions, first proved by Borodin \cite{Borodin2007}, later by Tingley \cite{Tingley2008} and Langer~\cite{Langer2013B}, can be obtained by letting
$u_i=z\ (0\leq i \leq h-1)$ and $u_{h}=1$  in 
\eqref{eq:MainUCylinder}.

\begin{thm}[Borodin\cite{Borodin2007}]\label{th:Borodin}
Let $\delta=(\delta_i)_{1\leq i\leq h}$ be a profile. 
Then the generating function for the cylindric partitions with profile $\delta$ is
\begin{equation*}
\sum_{\omega \in \PP{CP}_\delta}  z^{|\omega|}
=  
	\prod_{k\geq 0}  
	\prod_{t\in W_3(\delta)} \frac{1}{1-z^{hk+t}},
\end{equation*}
where $W_3(\delta)$ is the following multiset
\begin{align*}
	W_3(\delta)&:=\{h\} \cup \{ j-i : i < j,\ \delta_i > \delta_j\} \cup \{ h+i-j: i < j,\ \delta_i < \delta_j \}. 
\end{align*}
\end{thm}

A \emph{symmetric cylindric partition }$(\PP{SCP})$ with profile $\delta=(\delta_1, \delta_2, \ldots,  \delta_h)$  is a \PP{DSPP}
$$
\omega=(\la^h,\la^{h-1},\ldots,\la^{1},\la^0,\la^{1},\ldots,\la^{h-1},\la^h)
$$
with profile $(-\delta_h, -\delta_{h-1}, \ldots, -\delta_{2}, -\delta_1, \delta_1, \delta_{2}, \ldots,\delta_{h-1}, \delta_h)$.
Notice that a symmetric cylindric partition is always a cylindric partition, and
when $\la^h=\emptyset$, the \PP{SCP} $\omega$ becomes an \PP{SyPP}.
The size of the symmetric cylindric partition $\omega$ is defined by 
$$|\omega|=|\la^0|+2\sum_{i=1}^h |\la^h|.$$

\begin{thm}\label{th:symmetric_cylinder_plane_partitions}
 The generating function for symmetric cylindric partitions with profile $\delta$ is 
\begin{align} \label{eq:symmetric_cylinder_plane_partitions}
	&\sum_{\omega \in \PP{SCP}_\delta} z^{|\omega|} 
=\prod_{\substack{i < j \\ \delta_i > \delta_j}} \frac{1}{1-z^{2(j-i)}}
\Phi(\{ z^{2i-1}:  \delta_i=-1  \}) 
\\&
\qquad \qquad \times
\prod_{k\geq 1} \frac{\Phi(z^{(2h+1)k}(\{ z^{-2i+1}:  \delta_i=1  \}+\{ z^{2i-1}:  \delta_i=-1  \}))}{1-z^{(2h+1)k}}.\nonumber
\end{align}
\end{thm}
\begin{proof}
By \eqref{eq:MainUOpen}, we have
\begin{align*}
\sum_{\omega \in \PP{SCP}_\delta} z^{|\omega|} 
&=\sum_{\la^0,\la^h}z^{-|\la^0|}
\sum_{\omega' \in \PP{DSPP}_\delta(\la^0, \la^h)}  z^{2|\omega'|}
\\&=
\prod_{\substack{i < j \\ \delta_i > \delta_j}} \frac{1}{1-z^{2(j-i)}}\ 
\\&\times
\sum_{ \la,\mu}z^{2(h+1)|\mu|-|\la|}  \sum_{\gamma\subset \la,\mu}s_{\la/\gamma}(\{ z^{2i}:  \delta_i=-1  \})s_{\mu/\gamma}(\{ z^{-2i}:  \delta_i=1  \})
\\&
=\prod_{\substack{i < j \\ \delta_i > \delta_j}} \frac{1}{1-z^{2(j-i)}}\ 
\\&\times
\sum_{ \la,\mu}z^{(2h+1)|\mu|}  \sum_{\gamma\subset \la,\mu}s_{\la/\gamma}(\{ z^{2i-1}:  \delta_i=-1  \})s_{\mu/\gamma}(\{ z^{-2i+1}:  \delta_i=1  \}).
\end{align*}
By Lemma \ref{th:SchurSum:s2}, this is equal to the right hand side of \eqref{eq:symmetric_cylinder_plane_partitions}.
\end{proof}
%Let $g=2h+1$.
The right-hand side of the above identity can be further simplified. 
For each profile $\delta=(\delta_i)_{1\leq i\leq m-1}$, %so that the width of $\omega$, viewed as ribbon plane partition, is $2m-1$.
let $W_4(\delta)$ and $W_5(\delta)$ be the following multisets:
\begin{align*}
	W_4(\delta)&=\{2m-1\} \cup \{ 2i-1 \mid \delta_i=-1\} \cup \{ 2m-2i\mid \delta_{i}=1 \}; \\
	W_5(\delta) &=\{ 
	2i+2j-2\quad \mid 1\leq i<j \leq m-1,\ \delta_i=\delta_{j}=-1 \}   \\
	&\cup \{ 4m-2i-2j\ \mid 1\leq i<j \leq m-1,\ \delta_i=\delta_{j}=1 \}  \\
	&\cup \{ 2(2m-1)+2i-2j\mid 1\leq i<j \leq m-1,\ \delta_i<\delta_{j} \}  \\
  &\cup \{ 2j-2i \mid  1\leq i<j \leq m-1,\ \delta_i>\delta_{j} \}. 
\end{align*}
Then we obtain the following result.
\begin{thm}\label{th:symmetric_cylinder_PP}
The generating function for 
symmetric cylindric partitions with profile $\delta=(\delta_i)_{1\leq i\leq m-1}$
is
\begin{equation*}
\sum_{\omega \in \PP{SCP}_\delta}  z^{|\omega|}
=  
	\prod_{k\geq 0}  
	\left(\prod_{t\in W_4(\delta)} \frac{1}{1-z^{(2m-1)k+t}}\right) 
	\left(\prod_{t\in W_5(\delta)} \frac{1}{1-z^{2(2m-1)k+t}}\right).
\end{equation*}
\end{thm}

By the definitions of $W_4$ and $W_5$, it is easy to verify that
$|W_4(\delta)|=m$ and $|W_5(\delta)|=\binom{m-1}{2}$. Hence,
\begin{align}
\sum_{t \in W_4(\delta)} \frac{1}{2m-1}+ \sum_{t \in W_5(\delta)} \frac{1}{2(2m-1)}=\frac{m^2+m+2}{4(2m-1)}.
\end{align}
By Theorems \ref{th:Asy_Ribbon:multi} and \ref{th:symmetric_cylinder_PP} we obtain the following asymptotic formula for the number of \PP{SCP} with size $n$.
\begin{thm}
Let $m\geq 2$ be a positive integer and
$\delta=(\delta_j)_{1\leq j\leq m-1}$ be a profile of length $m-1$.
	Let $\PP{SCP}_\delta(n)$ denote the number of 
	symmetric cylindric partitions  with profile $\delta$ and size $n$. Then
\begin{align*}
\PP{SCP}_\delta(n)  \sim v \sqrt{\frac{1}{8\pi}} \frac{r^{b+1/4}}{n^{b+3/4}} \exp({\sqrt{nr}}), 
\end{align*}
where $r,b,v$ are given below:
\begin{align*}
r &= {\frac{(m^2+m+2)\pi^2}{6(2m-1)}},\\
b&= \sum_{t\in W_4(\delta) }(\frac{t}{2(2m-1)} - \frac 14)
  + \sum_{t\in W_5(\delta) } (\frac{t}{4(2m-1)} - \frac 14),\\
v
%&= \prod_{t\in W_4(\delta)} \left(\frac{\Gamma(t/(2m-1))}{\sqrt{(2m-1)\pi}} (\frac {2m-1}{2})^{t/(2m-1)} \right)  \prod_{t\in W_5(\delta)} \left(\frac{\Gamma(t/2(2m-1))}{\sqrt{2(2m-1)\pi}} ({2m-1})^{t/2(2m-1)} \right)
%
%= \prod_{t\in W_4(\delta)} \Gamma(\frac{t}{2m-1})
%\prod_{t\in W_5(\delta)} \Gamma(\frac{t}{2(2m-1)})\, \  \times 2^{-\frac{1}{2m-1}\sum_{t\in W_4(\delta)}t-\frac{1}{2}|W_5(\delta)|}
%\\& \times \pi^{ -\frac{1}{2} (|W_4(\delta)|+ W_5(\delta)|) } (2m-1)^{ -\frac{1}{2} (|W_4(\delta)|+ |W_5(\delta)|) +\frac{1}{2m-1}\sum_{t\in W_4(\delta)}t+\frac{1}{2(2m-1)}\sum_{t\in W_5(\delta)}t};\\
%
%	&= \prod_{t\in W_4(\delta)} \Gamma(\frac{t}{2m-1})
%	\prod_{t\in W_5(\delta)} \Gamma(\frac{t}{2(2m-1)})\, \  \times 2^{-\frac{1}{2m-1}\sum_{t\in W_4(\delta)}t-\frac{1}{4}(m^2-3m+2)}
%	\\& \times \pi^{ -\frac{1}{4} (m^2-m+2) } (2m-1)^{2b}.\\
	&= 
	2^{-\frac{1}{2m-1}\sum_{t\in W_4(\delta)}t-\frac{1}{4}(m^2-3m+2)}
	\pi^{ -\frac{1}{4} (m^2-m+2) } (2m-1)^{2b}\\
	& \qquad \times \prod_{t\in W_4(\delta)} \Gamma(\frac{t}{2m-1})
	\prod_{t\in W_5(\delta)} \Gamma(\frac{t}{2(2m-1)}).
\end{align*}
\end{thm}

{\it Remark}. The coefficient $r$ depends only on the width of the symmetric cylindric partitions, not on the profile itself, while the coefficient $b$ depends on the profile.
It is interesting to compare these phenomena with the asymptotic formula for 
cylindric partitions \cite{HanXiong2017}. There, the coefficient $b$ depends only on the width, and $r$ depends on the profile. 

\medskip

For example, consider the three symmetric cylindric partitions (SCPa)-(SCPc) given in Fig. 2. Their profiles,
generating functions  and asymptotic formulas are respectively:

%{\bf [TODO3: Done! add W4 W5 in the following cases]}

(a) Fig. 2, case SCPa. $\delta=  (-1, -1)$. 
$W_4(\delta)=\{1,3,5\}$ and $W_5(\delta)=\{4\}$. 
$$
\sum_{\omega \in \PP{SCPa}}  z^{|\omega|}
=
\prod_{k\geq 0}\frac{1}{(1-z^{5k+1})(1-z^{5k+3})(1-z^{5k+5})(1-z^{10k+4})},
$$
\begin{align*}
	\PP{SCPa}(n)\sim
	\Gamma(\frac15)\Gamma(\frac25)\Gamma(\frac35)2^{-\frac{19}{5}}
	5^{-\frac{3}{20}}\pi^{-\frac{9}{5}}(\frac73)^{\frac{7}{20}}
	\times  \frac{1}{n^{{17}/{20}}} 
	\exp\Bigl(\pi\sqrt{\frac{7n}{15}} \, \Bigr).
\end{align*}

\smallskip

(b) Fig. 2, case SCPb. $\delta=  (1, -1)$. 
$W_4(\delta)=\{3,4,5\}$ and $W_5(\delta)=\{2\}$. 
$$
\sum_{\omega \in \PP{SCPb}}  z^{|\omega|}
=
\prod_{k\geq 0}\frac{1}{(1-z^{5k+3})(1-z^{5k+4})(1-z^{5k+5})(1-z^{10k+2})},
$$
\begin{align*}
%& \psi_{\frac12,n}(\Gamma(\frac15)\Gamma(\frac35)\Gamma(\frac45)2^{-29/10}5^{3/5}\pi^{-2},\frac{7\pi^2}{15},\frac{3}{10})
%\\&=
	\PP{SCPb}(n)\sim
\Gamma(\frac15)\Gamma(\frac35)\Gamma(\frac45)
	2^{-\frac{22}{5}}5^{\frac {1}{20}}\pi^{-\frac{7}{5}}(\frac73)^{\frac{11}{20}}
	\times  \frac{1}{n^{{21}/{20}}} 
	\exp\Bigl(\pi\sqrt{\frac{7n}{15}} \, \Bigr).
\end{align*}

\smallskip

(c) Fig. 2, case SCPc. $\delta=  (1,1)$. 
$W_4(\delta)=\{2,4,5\}$ and $W_5(\delta)=\{6\}$. 
$$
\sum_{\omega \in \PP{SCPc}}  z^{|\omega|}
=
\prod_{k\geq 0}\frac{1}{(1-z^{5k+2})(1-z^{5k+4})(1-z^{5k+5})(1-z^{10k+6})},
$$
\begin{align*}
\PP{SCPc}(n)\sim
\Gamma(\frac25)\Gamma(\frac35)\Gamma(\frac45)
	2^{-\frac{21}{5}}5^{\frac{3}{20}}\pi^{-\frac{6}{5}}(\frac73)^{\frac{13}{20}}
\times  \frac{1}{n^{23/20}} 
	\exp\Bigl(\pi\sqrt{\frac{7n}{15}} \, \Bigr).
\end{align*}

% >>>
% 

\bibliographystyle{plain}
\bibliography{x10pp.bib}

\end{document}